\documentclass[11pt]{article}

\usepackage{booktabs}
\usepackage{multirow}
\usepackage{subfigure}
\usepackage{graphicx}           
\usepackage{caption}
\usepackage{tikz}
\usepackage{tikz-qtree} 
\usepackage{amsfonts}
\usepackage{mathdots}
\usepackage{amsmath}
\usepackage{amsthm} 
\usepackage{amssymb}
\usepackage{glossaries} 
\usepackage{mathtools}
\usepackage{enumerate} 
\usepackage{stmaryrd} 
\usepackage{xcolor}
\usepackage{bm}

\usepackage[margin=1.5in]{geometry}

\DeclarePairedDelimiter\abs{\lvert}{\rvert}%
\DeclarePairedDelimiter\norm{\lVert}{\rVert}%

\newcommand{\RR}{\mathbb{R}}

\newcommand{\onehalf}{1/2}
\newcommand{\vect}[1]{\boldsymbol{#1}}
\newcommand{\normalvect}{\boldsymbol{n}}
\newcommand{\vne}{\boldsymbol{n}_e}

\newcommand{\vecta}{\vect{a}}

\newcommand{\ddiv}{\text{div}\,}
\newcommand{\grad}{\nabla}

\newcommand{\bdry}{\partial\Omega}
\newcommand{\diribd}{\Gamma_D}
\newcommand{\neubd}{\Gamma_N}

\newcommand{\kpa}{\epsilon^{-1/2}\beta^{1/2}}

\newcommand{\kpaneg}{\epsilon^{1/2}\beta^{-1/2}}

\newcommand{\gammaK}{\gamma_{_K}}
\newcommand{\gammae}{\gamma_e}

\newcommand{\almax}{\alpha_{\max}}
\newcommand{\almin}{\alpha_{\min}}

\newcommand{\alrootn}{\alpha^{-\onehalf}}
\newcommand{\alphaK}{\alpha_{\!_K}}
\newcommand{\alphaKenb}{\alpha_{\!_{\kenb}}}

\newcommand{\gN}{g_{_N}}
\newcommand{\fbar}{\bar{f}}

\newcommand{\gbar}{\bar{g}_{_N}}

\newcommand{\ghatK}{\hat{g}_{_K}}

\newcommand{\rhatK}{\hat{r}_K}

\newcommand{\fluxmesh}{\hat{\boldsymbol{\sigma}}_{\sb{\mathcal{T}}}}
\newcommand{\fluxh}{\boldsymbol{\sigma}_{\sb{\mathcal{T}}}}

\newcommand{\fluxt}{\boldsymbol{\tau}}

\newcommand{\twonorm}[1]{\lVert #1 \rVert}

\newcommand{\ltwox}[1]{L^2(#1)}

\newcommand{\honeD}{H^1_D(\Omega)}
\newcommand{\hdiv}{H(\text{div};\Omega)}
\newcommand{\hdivk}[1]{H(\text{div};#1)}

\newcommand{\PiKkneg}{\Pi_K^{k-1}}

\newcommand{\enorm}[1]{{\left\vert\kern-0.25ex\left\vert\kern-0.25ex\left\vert #1 
    \right\vert\kern-0.25ex\right\vert\kern-0.25ex\right\vert}}
\newcommand{\enormx}[2]{{\left\vert\kern-0.25ex\left\vert\kern-0.25ex\left\vert #1 
    \right\vert\kern-0.25ex\right\vert\kern-0.25ex\right\vert}_{ #2 }}

\newcommand{\uapprox}{u_{\sb{\meshth}}}
\newcommand{\vapprox}{v_{\sb{\meshth}}}

\newcommand{\Vapprox}{V_{\meshth}}

\newcommand{\meshth}{{\mathcal{T}}}

\newcommand{\vertex}{\mathcal{N}}
\newcommand{\edge}{\mathcal{E}}
\newcommand{\edgek}{\mathcal{E}_K}
\newcommand{\intiedge}{\mathcal{E}_I}
\newcommand{\diriedge}{\mathcal{E}_D}
\newcommand{\neuedge}{\mathcal{E}_N}

\newcommand{\kenb}{K_e'}
\newcommand{\kepos}{K_e^+}
\newcommand{\keneg}{K_e^-}

\newcommand{\RT}{\mathrm{RT}}

\makeatletter
\let\oldabs\abs
\def\abs{\@ifstar{\oldabs}{\oldabs*}}
\let\oldnorm\norm
\def\norm{\@ifstar{\oldnorm}{\oldnorm*}}
\makeatother

\newtheorem{theorem}{Theorem}[section]
\newtheorem{lemma}{Lemma}[section]

\newtheorem{remark}{Remark}[section]  

\title{Hybrid A Posteriori Error Estimators for Conforming Finite Element Approximations to Stationary Convection-Diffusion-Reaction equations}
\author{Difeng Cai \and Zhiqiang Cai}

\date{}

\begin{document}
\maketitle

\begin{abstract}
We consider the a posteriori error estimation for convection-diffusion-reaction equations in both diffusion-dominated and convection/reaction-dominated regimes.
We present an explicit hybrid estimator, which, in each regime, is proved to be
reliable and efficient with constants independent of the parameters in the underlying problem.
For convection-dominated problems, the norm introduced by Verf{\"u}rth \cite{verf2005confusion} is used to measure the approximation error.
Various numerical experiments are performed to $(1)$ demonstrate the robustness of the hybrid estimator; $(2)$ show that the hybrid estimator is more accurate than the explicit residual estimator and is less sensitive to the size of reaction, even though both of them are robust.
\end{abstract}

\section{Introduction} 
The a posteriori error estimation has been an indispensable tool 
in handling computationally challenging problems.
For elliptic partial differential equations (PDEs) consisting of terms characterizing diffusion, reaction and convection,
the solution may display a strong interface singularity,
interior or boundary layers, etc.,
due to discontinuous coefficients or terms in significantly different scales, etc.
Consequently,
it is a challenging task to design a \emph{general} a posteriori error estimator
that is accurate and robust enough
to resolve local behaviors of the exact solution without spending much computational resource.

Explicit residual estimators are directly related to the error and have been well-studied since 1970s
(cf. \cite{babuska1978,babuska1987,verf1994,bernardi2000,petzoldt2002,verf1998reaction,verf1998confusion,verf2005confusion,verf2015confusion}).
They are easy to compute, applicable to a large class of problems, valid for higher order elements, etc.
More importantly, 
\emph{robust} residual estimators have been proposed for various problems.
For diffusion problems with discontinuous coefficient, \cite{bernardi2000,petzoldt2002} established robustness with respect to coefficient jump 
under the monotonicity or quasi-monotonicity assumption of the diffusion coefficient.
For singularly perturbed reaction-diffusion problems, 
Verf{\"u}rth \cite{verf1998reaction} pioneered a residual estimator that is robust with respect to the size of reaction.
For convection-dominated problems,
it is still under debate on how to choose a suitable norm to measure the approximation error
(cf. \cite{stynes2005confusion, verf1998confusion, sangalli2001confusion, sangalli2005norm, verf2005confusion, sangalli2008confusion}).
Sangalli \cite{sangalli2005norm,sangalli2008confusion} proposed a norm 
incorporating the standard energy norm and
a seminorm of order $1/2$ and developed the a posteriori error analysis in the one-dimensional setting.
Verf{\"u}rth \cite{verf2005confusion} introduced
a norm incorporating the standard energy norm and
a dual norm of the convective derivative.
With respect to this norm, the explicit residual estimator in \cite{verf2005confusion} was proved to be robust.
Moreover, it was shown in \cite{verf2015confusion} that the framework of residual estimators is applicable for various stabilization schemes.
Those developments make residual estimators competitive when dealing with challenging problems.
One drawback, however,
is that residual estimators tend to overestimate the true error by a large margin (cf. \cite{carstensen2010competition,localL2}).
This calls for the need of an estimator as general as the residual estimator but with improved accuracy.

Recovery-based estimators,
e.g., the Zienkiewicz–Zhu (ZZ) estimator and its variations
(cf. \cite{ZZ1987,ZZ1992,purdue1972,carstensen2004,nagazhang2005,bankxuzheng,ZZsafeguard}, etc.),
are quite popular in the engineering community.
However, unlike residual estimators,
the robustness of those estimators
with respect to issues like coefficient jump, dominated convection or reaction, etc.,
has not been emphasized or studied in detail yet (cf. \cite{ovall2006}).
On coarse meshes, it is known that ZZ-type estimators are in general unreliable,
and counterexamples can be easily constructed where the estimator vanishes but the true error is large (cf. \cite{ainsworthoden2000book,localL2}).
For linear elements, \cite{ZZsafeguard,ZZ2006} 
adds two additional terms (one of them is the element residual) to the ZZ estimator to ensure reliability on coarse meshes.
For higher order elements, however, a straightforward extension of the original ZZ estimator \cite{ZZ1987,ZZ1992}
usually fails and developing a viable estimator is nontrivial.
For example, Bank, Xu, and Zheng in \cite{bankxuzheng} recently introduced a recovery-based estimator for Lagrange triangular elements of degree $p$, 
and their estimator requires recovery of all partial derivatives of $p^{\rm th}$ order
instead of the gradient.

Recently, the so-called \emph{hybrid estimator} was introduced in \cite{localL2,caizhang2010} for diffusion problems with discontinuous coefficients.
The explicit hybrid estimator shares all advantages of the robust residual estimator \cite{bernardi2000,petzoldt2002}
and numerical results indicate that the hybrid estimator is more accurate than the residual estimator (cf. \cite{localL2}).
This opens a door of finding an alternative of the residual estimator with improved accuracy.
Thus one may ask if it is possible to construct hybrid estimators for more general problems
and if the hybrid estimator is still more accurate than the residual estimator.

In this manuscript,
we introduce the hybrid estimator as well as flux recoveries for convection-diffusion-reaction equations.
In diffusion-dominated regime,
the flux recovery as well as the hybrid estimator is a natural extension of the one in \cite{localL2}.
In convection/reaction-dominated regime,
the flux recovery in each element depends on the size of diffusion.
Roughly speaking,
in elements with resolved diffusion (see Section \ref{sub:flux1}),
the recovered flux is same to the diffusion-dominated case;
in elements where diffusion is not resolved,
inspired in part by the method of Ainsworth and Vejchodsk{\'y} \cite[Section 3.4]{ainsworth2011confusion},
the recovered flux is defined piecewisely in each element (see Section \ref{sub:flux2}).
The hybrid estimator in the convection/reaction-dominated regime is analogously defined as in \cite{localL2} with proper weights from \cite{verf2005confusion}.
In each regime, we prove that the hybrid estimator is equivalent to the robust residual estimator (for example, \cite{verf2005confusion} for the convection/reaction-dominated regime) and
then the robustness follows immediately from that of the residual estimator.
The hybrid estimator is explicit and valid for higher order elements.
Various numerical results show that, compared to the explicit residual estimator,
the hybrid estimator is more accurate and the corresponding effectivity index is less sensitive to the size of reaction.

The rest of the manuscript is organized as follows.
The model problems and finite element discretizations are introduced in Section \ref{sec:modelConfusion}. 
Section \ref{sec:etaConfusion} collects results on robust residual estimators.
After the flux recovery presented in Section \ref{sec:fluxConfusion},
the hybrid estimator is defined in Section \ref{sec:hybridConfusion} along with robust a posteriori error estimates.
Section \ref{sec:proofConfusion} gives the proof of the local equivalence between
the residual estimator and the hybrid estimator.
Numerical results are shown in Section \ref{sec:numericalConfusion}.

\section{Problems and Discretizations} 
\label{sec:modelConfusion}
Let $\Omega$ be a polygonal domain in $\mathbb{R}^d\,$ ($d=2,\, 3$) with Lipschitz boundary $\bdry$ consisting of two disjoint components $\diribd$ and $\neubd$.
By convention, assume that $\text{diam}(\Omega) = O(1)$. 
Consider the stationary convection-diffusion-reaction equation:
\begin{equation}
\label{eq:modelConfusion}
    \left\{ \begin{alignedat}{2}
    -\text{div}(\alpha\grad u) + \vecta\cdot\grad u + bu &= f,\quad && \text{in} \;\; \Omega,\\
  u &= 0, \quad && \text{on} \;\; \diribd,\\
  -\alpha\grad u\cdot\normalvect &= \gN, \quad && \text{on} \;\; \neubd,
    \end{alignedat} \right.
\end{equation}
with $\alpha(x) \geq \delta$, for almost all $x \in\Omega$ and for some constant $\delta>0$.
Assume that:
\begin{enumerate}[({A}1)]
    \item $\vecta\in W^{1,\infty}(\Omega)^d$ and  $b\in L^{\infty}(\Omega)$;
    \item there are two constants $\beta\geq 0$ and $c_b\geq 0$,
        independent of $\alpha$, such that 
\[
    b-\frac{1}{2}\ddiv\, \vecta \geq \beta \;\text{in}\; \Omega
        \quad \text{and} \quad \twonorm{b}_{\infty} \leq c_b\beta;
\]
    \item $\text{meas}(\diribd)>0$ and $\diribd$ contains the inflow boundary
\[
   \{ x\in\bdry: \vecta(x)\cdot \normalvect(x) < 0 \}.
\]
\end{enumerate}
Depending on the magnitude of the $\alpha$ (with respect to $\vecta$ and $b$),
two regimes are studied in this paper:
\begin{enumerate}
    \item \textbf{diffusion-dominated regime}: 
    there exists a constant $C_b \geq 0$ such that 
\[
    |\vecta(x)/\alpha(x)|\leq C_b
    \quad\text{and}\quad |b(x)/\alpha(x)|\leq C_b 
\quad \text{for almost all } x\in\Omega;
\]
    \item \textbf{convection/reaction-dominated regime}: 
    $\alpha(x)\equiv\epsilon\ll 1$ for a constant $\epsilon>0$.
    This is the so-called singularly perturbed problem.
\end{enumerate}
Let
\[
    \honeD:= \{v\in H^1(\Omega):v|_{\diribd}=0\}.
\]
Define the bilinear form on $\honeD$ by
\[
    B(u,v) := (\alpha\grad u, \grad v) + (\vecta\cdot\grad u, v) + (bu,v) ,\quad \forall\, u,v\in \honeD,
\]
where $(\cdot,\cdot)_S$ denotes the $L^2$ inner product on set $S$ and the subscript $S$ is omitted when $S=\Omega$. The $L^2$ norm on $S$ is denoted by $\twonorm{\cdot}_S$.

The weak formulation of \eqref{eq:modelConfusion} is to find $u\in\honeD$ such that
\begin{equation}
\label{eq:weakform}
    B(u,v) = (f,v) - (\gN,v)_{\neubd}, \quad \forall\, v\in\honeD.
\end{equation}
It follows from integration by parts and the assumptions in (A2) and (A3) that, for any $v\in\honeD$,
\begin{equation}
\label{eq:bgraduu}
    (\vecta\cdot\grad v,v)+(bv,v) = 
\frac{1}{2}(v^2, \vecta\cdot\normalvect)_{\neubd} + (v^2,b-\frac{1}{2}\ddiv\,\vecta) 
\geq \beta\twonorm{v}^2,
\end{equation}
where $\normalvect$ denotes the unit outward vector normal to $\neubd$.
The energy norm induced by $B(\cdot,\cdot)$ is defined by
\[
    \enorm{v} = \left(\twonorm{\alpha^{1/2}\grad v}^2 + \beta\twonorm{v}^2\right)^{1/2},
    \quad \forall\, v\in\honeD,
\]
where $\enorm{\cdot}_S$ denotes the energy norm over $S$ and
the subscript $S$ is omitted when $S=\Omega$.

Let $\meshth$ be a regular triangulation of $\Omega$ (see, e.g., \cite{ciarletbook}).
Define the following sets associated with the triangulation $\meshth$:
\[
\begin{aligned}
    \vertex &: \text{ the set of all vertices},\\
    \edge   &: \text{ the set of all edges} (d=2) / \text{faces} (d=3),\\
    \intiedge &: \text{ the set of all interior edges} (d=2) / \text{faces} (d=3),\\
    \diriedge&: \text{ the set of all edges} (d=2) / \text{faces} (d=3) \text{ on } \diribd,\\
    \neuedge&: \text{ the set of all edges}  (d=2) / \text{faces} (d=3) \text{ on } \neubd,\\
    \edge_K &: \text{ the set of edges} (d=2) / \text{faces} (d=3) \text{ in an element } K\in\meshth.
\end{aligned}
\]
For a simplex $S\in\meshth\cup\edge$,
denote by $|S|$ and $h_S$ its measure and diameter, respectively.
Denote by $R_K$ the inradius of $K\in\meshth$.
The shape regularity of the triangulation requires the existence of a generic constant $C_0 > 1$ 
such that
\begin{equation}
\label{eq:shaperegular}
   h_K\leq C_0 R_K,\quad \forall\, K\in\meshth
\end{equation}
holds true for each mesh $\meshth$ in the adaptive mesh refinement procedure.

We associate each $e\in\edge$ with a unit normal $\vne$, which is chosen as the unit outward normal if $e\subset \bdry$.
Denote by $\kepos$ and $\keneg$ the two elements
sharing $e$ such that the unit outward normal of $\kepos$ on $e$ coincides with $\vne$.
For $e\in\intiedge\cap\edgek$, let $\kenb$ denote the element next to $K$
sharing $e$ in common. 
Let $\omega_e$ be the union of elements adjacent to $e\in\edge$
and $\omega_K$ be the union of elements that share at least one face with $K\in\meshth$.
Unless otherwise stated,
$\normalvect$ always denotes the unit outward normal vector on $\partial K$.

For $k=0,1,2,\dots$,
let $P_k(S)$ denote the set of polynomials of degree at most $k$ on $S\in\meshth\cup\edge$
and $\Pi_{S}^{k}$ denote the $L^2(S)$-projection onto $P_k(S)$.
Define the conforming finite element space of order $k$ ($k\geq 1$) by
\[
\begin{aligned}
    \Vapprox &:= \{v\in C(\Omega): v|_K\in P_k(K),\; \forall K\in\meshth, \text{ and } v|_{\diribd} = 0\}.
\end{aligned}
\]
For each $K\in\meshth$, the Raviart-Thomas space of index $k-1$ ($k\geq 1$) is
\begin{equation*}
    \RT_{k-1}(K) := \left\{ \fluxt\in \ltwox{K}^d: \fluxt = \vect{p}+\vect{x}q,\;
        \vect{p}\in  P_{k-1}(K)^d,\; q\in P_{k-1}(K) \right\}.
\end{equation*}
The projected data $\fbar$ and $\gbar$ are defined by
\[
    \fbar|_K := \PiKkneg f,\quad \forall\, K\in\meshth
    \quad \text{and} \quad
    \gbar|_e := \PiKkneg \gN, \quad \forall\, e\in\neuedge,
\]
respectively.

The standard finite element approximation for problem \eqref{eq:modelConfusion} is to find $\uapprox\in\Vapprox$ such that
\begin{equation}
\label{eq:feformDiff} 
    B(\uapprox,\vapprox) = (f,\vapprox)-(\gN,\vapprox)_{\neubd},\quad \forall\, \vapprox\in \Vapprox.
\end{equation}

In the case that the convection is dominant, one often adds a stabilization term along the convective direction. 
For example, the so-called SUPG method in \cite{SUPG1992}
is to find $\uapprox\in\Vapprox$ such that
\begin{equation}
\label{eq:feformConv}
    B_{\delta}(\uapprox,\vapprox) = l_{\delta}(\vapprox),\quad \forall\, \vapprox\in\Vapprox,
\end{equation}
where the stabilized bilinear and linear forms are given by
\[
    B_{\delta}(\uapprox,\vapprox) =
    B(\uapprox,\vapprox) + \sum_{K\in\meshth} 
        \delta_K(-\epsilon\Delta\uapprox+\vecta\cdot\grad\uapprox+b\uapprox, \vecta\cdot\grad\vapprox)_K
\]
for all $\uapprox,\vapprox\in \Vapprox$ and 
\[
    l_{\delta}(\vapprox) = (f,\vapprox)-(\gN,\vapprox)_{\neubd} + \sum_{K\in\meshth} \delta_K(f, \vecta\cdot\grad\vapprox)_K
\]
for all $\vapprox\in\Vapprox$,
respectively.
Here, the stabilization parameters $\delta_K$ are nonnegative and satisfy
\[
    \delta_K\twonorm{\vecta}_{L^{\infty}(K)} \leq Ch_K, \quad \forall\, K\in\meshth.
\]
To measure the convective derivative,
the following dual norm was used in \cite{verf2005confusion}:
\[
    \enorm{\phi}_* = \sup_{v\in \honeD\backslash\{0\}} \dfrac{\langle\phi,v\rangle}{\enorm{v}},
\]
where $\phi$ is in the dual space of $\honeD$ and $\langle \cdot,\cdot \rangle$ denotes the duality pairing.
Following \cite{verf2005confusion},
the dual norm will be combined with the energy norm to measure the approximation error 
in the convection-dominated regime.

To keep the exposition simple, we ignore data oscillation in coefficients
by further assuming that
for each $K\in\meshth$,
$\alphaK := \alpha|_K$, $\vecta|_K$, and $b|_K$ are constants.
The algorithm and analysis remain valid without this assumption 
if we replace those quantities with their proper projections 
and add the corresponding oscillation error in the estimates.

Define
\[
    \alpha_e := \max_{K\subseteq\omega_e} \alphaK,\quad
    \almax := \max_{K\in\meshth} \alphaK, \quad \text{and} \quad \almin:=\min_{K\in\meshth} \alphaK.
\]

For diffusion-dominated case,
we are interested in the case where $\alpha$ may be discontinuous 
and the hybrid estimator is robust with respect to the discontinuity.
For convection/reaction-dominated case ($\alpha=\epsilon\ll 1$), 
we design a hybrid estimator that is robust with respect to $\epsilon$ and $\beta$
in appropriate norms.


\section{Residual Estimator} 
\label{sec:etaConfusion}
Let $\fluxh=-\alpha\grad \uapprox$ be the numerical flux,
then the element residual $r_K\in\ltwox{K}$ and the flux jump across $e\in\edge$ are given by
\begin{equation}
\label{eq:rK} 
   r_K := \fbar-\vecta\cdot\grad\uapprox-b\uapprox-\ddiv\fluxh
\end{equation}
and
\begin{equation}
\label{eq:jeconfusion}
    j_e := \begin{cases}
        (\fluxh|_{\kepos}-\fluxh|_{\keneg})\cdot\vne, &\text{if } e\in\intiedge,\\[2mm]
        \fluxh\cdot\normalvect-\gbar, &\text{if } e\in\neuedge,\\[2mm]
        0, &\text{if } e\in\diriedge,
    \end{cases}
\end{equation}
respectively.
For $S\in \meshth\cup\edge$, define the weight $\gamma_{_S}$ as below
\begin{equation}
\label{eq:gamma}
    \gamma_{_S} = 
\begin{cases}
    1, &\text{in diffusion-dominated regime},\\
    \min \{1, h_S^{-1}\alpha_{_S}^{1/2}\beta^{-1/2}\}, &\text{in convection/reaction-dominated regime}.
\end{cases}
\end{equation}
The residual estimator is defined by
\begin{equation}
\label{eq:etaConfusion}
    \eta = \left( \sum_{K\in\meshth} \eta_K^2  \right)^{1/2} \quad\text{with}\quad
    \eta_{K}^2 = \gammaK^2 h_K^2\alphaK^{-1}\twonorm{r_K}_K^2 
        + \frac{1}{2}\sum_{e\in\edgek}\gammae h_e\alpha_e^{-1}\twonorm{j_e}_e^2.
\end{equation}
Note that for diffusion-dominated problems,
the $\eta$ is a simple extension of the one in \cite{bernardi2000} or \cite{petzoldt2002} for pure diffusion problems;
for convection/reaction-dominated problems,
the $\eta$ is introduced by Verf{\"u}rth in \cite{verf2005confusion}.
In fact,
$\gamma_{_S}h_S\epsilon_{_S}^{-1/2}$ is same to the weight defined in \cite[Eq.(3.4)]{verf2005confusion}.
Here, $\gamma_{_S}$ is additional weight needed for convection/reaction-dominated problems.
\begin{remark}
    Note that in the case of vanishing reaction $($i.e., $\beta=0$$)$, $\gamma_{_S}=1$
    and consequently the weights in $\eta_K$ \eqref{eq:etaConfusion} are identical to the ones for diffusion problems \cite{bernardi2000}.
\end{remark}

In the remainder of this section, we will discuss reliability and efficiency bounds of the estimator $\eta$.

\subsection{Convection/reaction-dominated regime} 
\label{sub:residual in Convection-dominated}
The global reliability and efficiency bounds of the $\eta$ were established by Verf{\"u}rth in \cite[Theorem 4.1]{verf2005confusion}.
The reliability and efficiency constants are uniform with respect to $\epsilon$ and $\beta$.
For reader's convenience, they are cited below.

Define the data oscillation on $K$ by
\[
    \Theta_K^2 = \gammaK^2 h_K^2\alphaK^{-1}\twonorm{f-\fbar}_K^2 + 
        \sum_{e\in\edgek\cap\neuedge} \gammae h_e\alpha_e^{-1}\twonorm{\gN-\gbar}_e^2,
    \quad \forall\, K\in\meshth.
\]
\begin{theorem}
\label{thm:etaConvDiff}
    Let $u$ be the solution of \eqref{eq:weakform} and $\uapprox$ be the solution of \eqref{eq:feformDiff} or \eqref{eq:feformConv}.
    Let $\eta_{K}$ be defined in \eqref{eq:etaConfusion}.
    Then
\[
    \enorm{u-\uapprox}^2 + \enorm{\vecta\cdot\grad(u-\uapprox)}_*^2 \leq C_1 
     \sum_{K\in\meshth} \left( \eta_K^2 + \Theta_K^2  \right) 
\]
and 
\[
    \sum_{K\in\meshth} \eta_K^2  \leq C_2 \left( \enorm{u-\uapprox}^2 + 
    \enorm{\vecta\cdot\grad(u-\uapprox)}_*^2 +  \sum_{K\in\meshth} \Theta_K^2 \right),
\]
where the constants $C_1$ and $C_2$ are independent of $\epsilon,\beta$, and any mesh-size.
\end{theorem}
Here and thereafter, we will use $C$ with or without subscripts to denote a generic nonnegative constant, 
possibly different at different occurrences, that is independent of any mesh-size and the problem parameters: either $\epsilon$ and $\beta$ for dominant convection/reaction, or $\almax/\almin$ for the dominant diffusion,
but may depend on the shape parameter of mesh $\meshth$ and on the polynomial degree $k$.

The reaction-dominated diffusion problem, 
i.e., with $\vecta=0$, $b=1$, and $\diribd=\bdry$,
corresponds to the singularly perturbed reaction-diffusion equation:
\begin{equation}
\label{eq:modelReaction}
    \left\{ \begin{alignedat}{2}
    -\epsilon \Delta  u + u &= f \quad && \text{in} \;\; \Omega,\\
    u &= 0 \quad && \text{on} \;\; \bdry.
    \end{alignedat} \right.
\end{equation}
The assumptions in Section \ref{sec:modelConfusion} are fulfilled with $\beta=c_b=1$.
For the finite element approximation in \eqref{eq:feformDiff} to the problem in \eqref{eq:modelReaction},
the residual estimator $\eta$ in \eqref{eq:etaConfusion}
is proved by Verf{\"u}rth in \cite{verf1998reaction}
to be globally reliable (see Theorem \ref{thm:etaConvDiff});
moreover, it is not only globally but also locally efficient.
\begin{theorem}
\label{thm:etaReaction}
    Let $u$ and $\uapprox$ be the respective exact and finite element solutions
    of the reaction-diffusion equation in \eqref{eq:modelReaction}.
    The residual error indicator $\eta_K$ in \eqref{eq:etaConfusion}
    satisfies 
\[
    \eta_K^2 \leq C \left( \enormx{u-\uapprox}{\omega_K}^2 +
             \sum_{K'\subset \omega_K} \Theta_{K'}^2 \right).
\]
\end{theorem}

\subsection{Diffusion-dominated regime} 
\label{sub:residual in Diffusion-dominated}
For the diffusion-dominated case,
following the pure diffusion case in \cite{bernardi2000} or \cite{petzoldt2002},
this section establishes global reliability and local efficiency of the estimator $\eta$.
These estimates are proved to be robust in terms of the discontinuity of the diffusion coefficient $\alpha$.

To this end, let $u$ be the exact solution in \eqref{eq:weakform} and $\uapprox$ be the finite element solution in \eqref{eq:feformDiff}.
An alternative and often used expression of the element residual $r_K$ in \eqref{eq:rK} is given 
in terms of the true error and the data oscillation:
\begin{equation}
\label{eq:rKdiv}
    r_K = -\ddiv (\alpha\grad(u-\uapprox))+\vecta\cdot\grad(u-\uapprox)+b(u-\uapprox) + \fbar-f.
\end{equation}
\begin{theorem}
\label{thm:etarelia}
    Under the monotonicity assumption \textnormal{\cite[Hypothesis 2.7]{bernardi2000}} of $\alpha$, the residual estimator $\eta$ satisfies the following reliability estimate: there exists a positive constant $C$ such that 
    \[
        \enorm{u-\uapprox}^2 \leq C \left( \eta^2 + \sum_{K\in\meshth} \Theta_K^2 \right).
    \]
\end{theorem}
\begin{proof}
    Let $w:=u-\uapprox$.
    To prove the reliability bound,
    according to \cite{bernardi2000},
    it suffices to derive an estimate of the error in the form below:
    \begin{equation}
    \label{eq:errorfunctional}
    \begin{aligned}
        \enorm{w}^2 \leq & \sum_{K\in\meshth} (r_K,w-w_{\meshth})_K + \sum_{e\in\edge} (j_e, w-w_{\meshth})_e \\
        & + \sum_{K\in\meshth} (\fbar-f,w-w_{\meshth})_K 
         + \sum_{e\in\neuedge} (\gbar-\gN,w-w_{\meshth})_e,
        \quad \forall\, w_{\meshth}\in \Vapprox.
    \end{aligned}
    \end{equation}

    To do so, it follows from integration by parts, \eqref{eq:rKdiv}, and 
    the error equation $B(w,w_{\meshth})=0$ for all $w_{\meshth}\in\Vapprox$ that
    \[
    \begin{aligned}
        (\alpha\grad w,\grad w) =& (\alpha\grad w,\grad(w-w_{\meshth})) + (\alpha\grad w,\grad w_{\meshth})\\
            =& \sum_{K\in\meshth} (r_K,w-w_{\meshth})_K + \sum_{e\in\edge} (j_e,w-w_{\meshth})_e
                - (\vecta\cdot\grad w, w) - (bw,w) \\
        &+ \sum_{K\in\meshth} (\fbar-f,w-w_{\meshth})_K 
         + \sum_{e\in\neuedge} (\gbar-\gN,w-w_{\meshth})_e,
    \end{aligned}
    \]
which, together with \eqref{eq:bgraduu},
yields \eqref{eq:errorfunctional}.
This completes the proof of the theorem.
\end{proof}

\begin{theorem}
\label{thm:etaeff}
    There exists a constant $C$ such that 
    \[
        \eta_K \leq C \left( \enorm{\uapprox-u}_{\omega_K} 
        + \sum_{K'\subseteq \omega_K} \Theta_{K'} \right), 
        \quad \forall\, K \in \meshth.
    \]
\end{theorem}
\begin{proof}
The proof is essentially same to the one in \cite{bernardi2000} for diffusion problems, with the only additional observation that
    \begin{equation*}
        h_K\alphaK^{-1/2}\twonorm{\vecta\cdot\grad w}_K
        \leq h_K C_b \twonorm{\alpha^{1/2}\grad w}_K
        \leq C \enorm{w}_K
    \end{equation*}
and
    \begin{equation*}
        h_K\alphaK^{-1/2}\twonorm{bw}_K
        \leq h_K C_b^{1/2}c_b^{1/2}\beta^{1/2}\twonorm{w}_K
        \leq C \enorm{w}_K
    \end{equation*}
according to the assumptions in Section \ref{sec:modelConfusion}.
\end{proof}

\section{Flux Recovery} 
\label{sec:fluxConfusion}
We show in this section how to recover a suitable flux in $\hdiv$, denoted by $\fluxmesh$, 
such that the resulting hybrid estimator is robust.
Same to \cite{localL2}, the recovered normal component $\fluxmesh|_K\cdot\normalvect=\ghatK\in\ltwox{\partial K}$ on $e\in\edgek$ is defined as a weighted average of normal components:
\begin{equation}
\label{eq:ghatK}
    \ghatK|_e := 
\begin{cases}
    \lambda_{K,e}\fluxh|_K\cdot\normalvect+(1-\lambda_{K,e})\fluxh|_{\kenb}\cdot\normalvect, &\text{if } e\in \edgek\cap\intiedge,\\[2mm]
        \gbar, &\text{if } e\in\edgek\cap\neuedge,\\[2mm]
        \fluxh\cdot\normalvect, &\text{if } e\in\edgek\cap\diriedge,
    \end{cases}
\end{equation}
where $\kenb$ is the element sharing $e\in\edgek$ and the average weight is given by
\[
    \lambda_{K,e} = \frac{\alphaK^{-1}h_K}{\alphaK^{-1}h_K+\alphaKenb^{-1}h_{\kenb}}.
\]

For the diffusion-dominated case, $\fluxmesh$ is defined in \eqref{eq:fluxDiffConv}.
For convection/reaction-dominated case, the form of $\fluxmesh$ in an element $K$ depends on the size of the element $K$. 
Its construction remains the same as in \eqref{eq:fluxDiffConv} if $R_K$ is relatively small.
Otherwise, it is essentially the numerical flux with transitional regions to be in $\hdivk{K}$.

\subsection{Diffusion-dominated regime} 
In the diffusion-dominated regime, the flux recovery is a natural extension of the one in \cite{localL2} for pure diffusion problems.
We define $\fluxmesh|_K\in \RT_{k-1}(K)$ by assigning the degrees of freedom:
\begin{equation}
\label{eq:fluxDiffConv}
    \left\{ \begin{alignedat}{1}
    \fluxmesh\cdot\normalvect&= \ghatK \quad \text{on } \partial K,\\[2mm]
    \ddiv\fluxmesh &= \PiKkneg(\fbar-\vecta\cdot\grad\uapprox-b\uapprox) + J_K  \quad \text{in } K,\\[2mm]
    \int_K \fluxmesh\cdot\vect{q} dx &= \int_K \fluxh \cdot\vect{q} dx \quad \forall\, \vect{q}\in \mathcal{Q}_{k-2}(K), \quad k\geq 2,
    \end{alignedat} \right.
\end{equation}
where
\begin{equation}
\label{eq:JKDiffConv}
   J_K := |K|^{-1} \left( \int_{\partial K} \ghatK ds + \int_K \vecta\cdot\grad\uapprox + b\uapprox - f dx \right)
\end{equation}
and 
\begin{equation}
\label{eq:Qk}
   \mathcal{Q}_{k-2}(K) := \{ \vect{q}\in P_{k-2}(K)^d: (\vect{q},\grad p)_K=0,\, \forall\, p\in P_{k-1}(K) \}.
\end{equation}
It is easily seen that $\fluxmesh\in\hdiv$
and $\fluxmesh$ coincides with the one in \cite[Eq.(3.14)]{localL2} when $\vecta$ and $b$ vanish.

\subsection{Flux recovery in convection/reaction-dominated regime} 
In the case that $\alpha=\epsilon$ is a constant,
the weight is $\lambda_{K,e}=\frac{h_K}{h_K+h_{\kenb}}$.
This may be simplified as $\lambda_{K,e}=\frac{1}{2}$ for all $K\in\meshth$
if the size of each element is close to sizes of its neighboring elements.
With given normal components of the recovered flux on each face,
the construction on each element $K\in\meshth$ depends on the size of $K$.

\subsubsection{Flux recovery in $K$ with $R_K\leq \kpaneg$} 
\label{sub:flux1}
On element $K\in\meshth$ with $R_K\leq \kpaneg$,
the recovered flux $\fluxmesh|_K\in \RT_{k-1}(K)$ is simply defined as in \eqref{eq:fluxDiffConv}.

\begin{remark}
    The case of absent reaction corresponds to $\beta=0$, or equivalently,
     $\beta^{-1/2} = \infty$.
    Consequently,
    $R_K\leq \epsilon^{1/2}\beta^{-1/2}$ holds true for all $K\in\meshth$.
    Therefore, 
    the recovered flux in each element is given by \eqref{eq:fluxDiffConv}. 
\end{remark}

\subsubsection{Flux recovery in $K$ with $R_K > \kpaneg$} 
\label{sub:flux2}
For simplicity, here we restrict our attention to $\Omega \subset \RR^2$.
For the three dimensional case, the partition of the tetrahedron follows \cite{ainsworth2013confusion} and the piecewise definition of the recovered flux shares similar idea as the two dimensional case.

For element $K\in\meshth$ with $R_K > \kpaneg$, 
we will define the recovered flux $\fluxmesh\in\hdivk{K}$ piecewisely.
To this end, as in \cite{ainsworth2011confusion}, $K$ is first partitioned into a triangle $K_{\Delta}$ and three trapezoids illustrated in Figure \ref{fig:partitionK}.
The edges of the triangle $K_{\Delta}$ are parallel to corresponding edges of $K$
and the distance between each pair of parallel edges from $K_{\Delta}$ to $K$ is equal to $\kpaneg$.
Each trapezoid is further partitioned into a rectangle and two triangles.
Figure \ref{fig:trapezoid} illustrates the partition of the trapezoid on $e\in\edgek$
into a rectangle, denoted by $Q_e$, and two triangles, denoted by $T_e$ and $S_e$, respectively.

\begin{figure}[htbp]
\centering
\begin{minipage}[t]{0.5\textwidth}
\centering
\includegraphics[width=.9\textwidth]{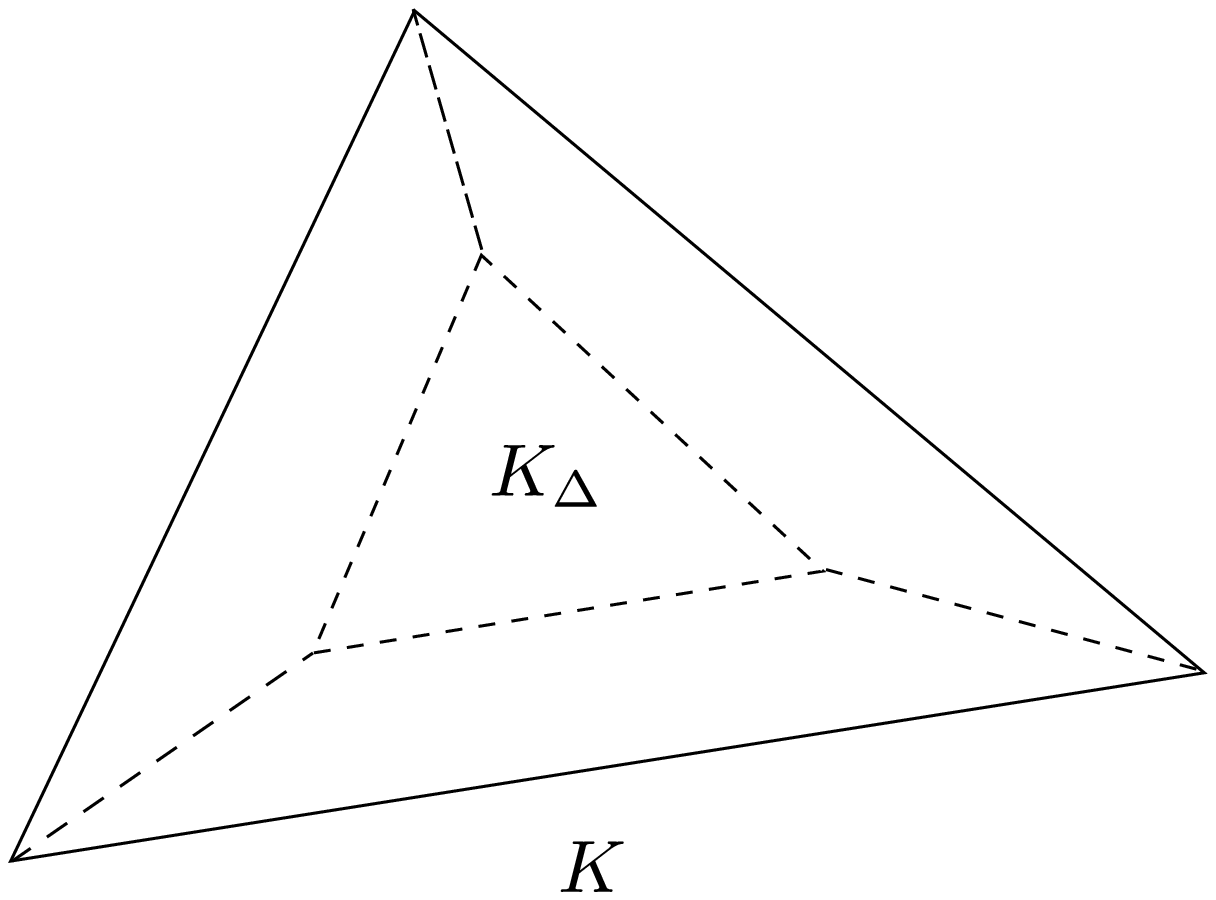}
\caption{partition of element $K$}
\label{fig:partitionK}
\end{minipage}\hfill
\begin{minipage}[t]{0.5\textwidth}
\includegraphics[width=1.0\textwidth]{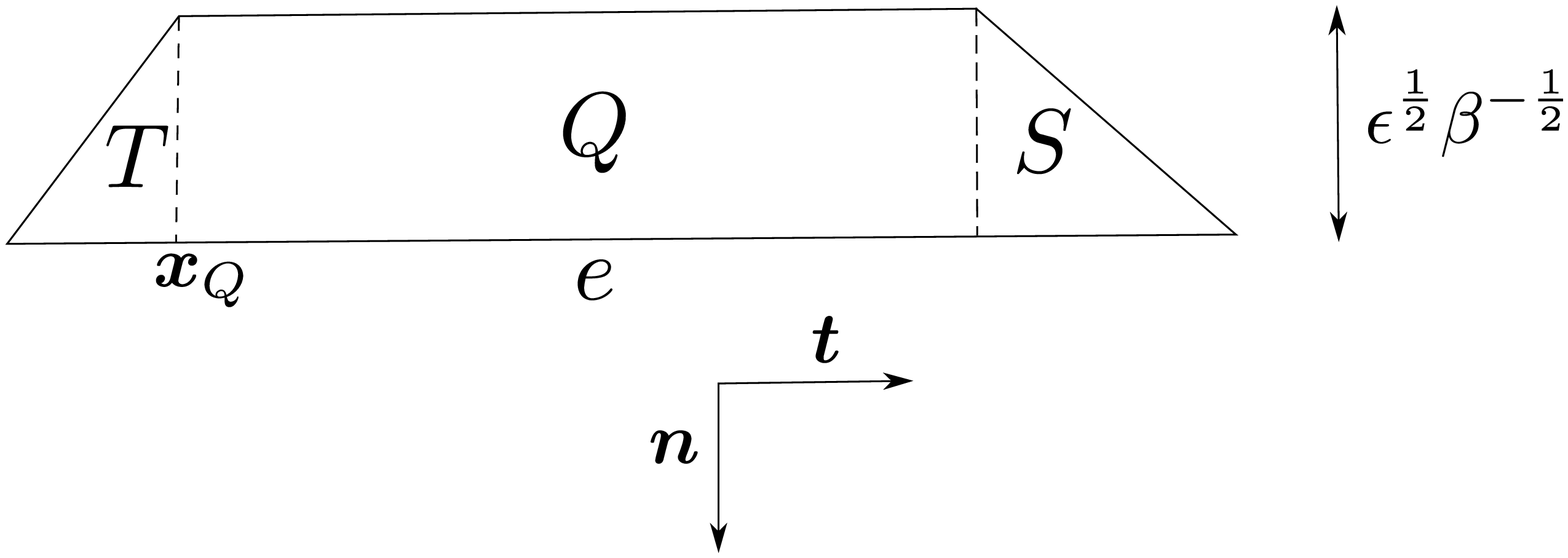}
\caption{partition of a trapezoid in $K$}
\label{fig:trapezoid}
\end{minipage}
\end{figure}


With the above partition $\mathcal{T}_K$ of $K$, to guarantee $\fluxmesh\in\hdivk{K}$, its normal component on element interfaces of $\mathcal{T}_K$ is defined to be that of the numerical flux.
On the triangle $K_{\Delta}$, the recovered flux is chosen to be the numerical flux:
\begin{equation}
\label{eq:flux2K}
    \fluxmesh = \fluxh \quad \text{in }  K_{\Delta}.
\end{equation}

On each rectangle, $\fluxmesh$ equals to $\fluxh$ plus a correction from an edge of $K_{\Delta}$ to its parallel edge of $K$ such that the normal components of $\fluxmesh$ is equal to the preassigned normal components on these two edges.
Specifically, let $\vect{x}_{Q_e} = \partial T_e \cap \partial Q_e \cap \partial K$ and $\vect{t}$ be a unit vector normal to $\normalvect$ (see Figure \ref{fig:trapezoid}), the recovered flux $\fluxmesh$ is given by:
    \begin{equation}
    \label{eq:flux2Q}
       \fluxmesh = \fluxh + (\kpa y -1) \left( \fluxh(\vect{x}_{Q_e}+x\vect{t})\cdot\normalvect - \ghatK(\vect{x}_{Q_e}+x\vect{t})  \right) \normalvect,
    \end{equation}
where $x=(\vect{x}-\vect{x}_{Q_e})\cdot\vect{t}$ and $y=(\vect{x}_{Q_e}-\vect{x})\cdot\normalvect$ are local coordinates of point $\vect{x}\in Q_e$.

On each triangle $\omega= T_e$ or $S_e$, $\fluxmesh|_{\omega}\in \RT_{k-1}(\omega)$ is defined similar to that in \eqref{eq:fluxDiffConv} (see \cite{localL2}):
\begin{equation}
\label{eq:flux2w}
    \left\{ \begin{alignedat}{2}
    \fluxmesh\cdot\normalvect_{\omega}&= \fluxh\cdot\normalvect_{\omega} \quad &&\text{on } \partial\omega \backslash\partial K,\\[1.5mm]
    \fluxmesh\cdot\normalvect_{\omega}&= \ghatK \quad &&\text{on } \partial \omega \cap \partial K,\\[1.5mm]
    \ddiv\fluxmesh &= \ddiv\fluxh + J_{\omega}  \quad &&\text{in } \omega,\\[1.5mm]
    \int_{\omega} \fluxmesh\cdot\vect{q} dx &= \int_{\omega} \fluxh \cdot\vect{q} dx \quad &&\forall\, \vect{q}\in \mathcal{Q}_{k-2}(\omega),\quad k\geq 2,
    \end{alignedat} \right.
\end{equation}
where
$\normalvect_{\omega}$ denotes the unit outward normal for $\omega$,
\begin{equation}
\label{eq:Jw}
   J_{\omega} := |\omega|^{-1} \int_{\partial\omega\cap \partial K} \left(\ghatK - \fluxh|_K\cdot\normalvect_{\omega}\right) ds  
\end{equation}
and the space $\mathcal{Q}_{k-2}(\omega)$ is defined as in \eqref{eq:Qk}.

Inside $K$, it is easy to verify that the normal component of $\fluxmesh$ is continuous
over $\partial\omega\backslash\partial K$ and $\partial Q_e\cap\partial K_{\Delta}$,
whose value is actually equal to that of $\fluxh$.
Hence $\fluxmesh\in\hdivk{K}$.
On $\partial K$, the normal component of $\fluxmesh$ is always given by $\ghatK$.
Therefore, we see that $\fluxmesh\in\hdiv$.




\section{Hybrid Estimator} 
\label{sec:hybridConfusion}
The hybrid estimator is defined as
\begin{equation}
\label{eq:xiConfusion}
    \xi = \left(\sum_{K\in\meshth} \xi_K^2\right)^{1/2} \quad \text{with}\quad 
    \xi_K^2 = \twonorm{\alrootn(\fluxmesh-\fluxh)}_K^2 + \gammaK^2 h_K^2\alphaK^{-1}\twonorm{\rhatK}_K^2,
\end{equation}
where $\gammaK$ is defined in \eqref{eq:gamma} and the modified element residual $\rhatK$ is given by
\begin{equation}
\label{eq:rhatK}
    \rhatK = \fbar-\ddiv\fluxmesh-\vecta\cdot\grad\uapprox-b\uapprox = r_K-\ddiv(\fluxmesh-\fluxh),
    \quad \forall \, K\in\meshth.
\end{equation}


{To establish the reliability and the efficiency of the hybrid estimator $\xi$,
we use the following local equivalence between $\xi$ and $\eta$,
which is proved in Section \ref{sec:proofConfusion}.}
\begin{lemma}
\label{lm:equivalenceConfusion}
Let $\eta_K$ be the residual estimator in \eqref{eq:etaConfusion}
and $\xi_K$ be the hybrid estimator in \eqref{eq:xiConfusion}, respectively.
Then there exist positive constants $C_1$ and $C_2$ such that
\begin{equation}
\label{eq:equivalenceConfusion}
   C_1 \xi_K \leq \eta_K \leq C_2 \sum_{K'\subset \omega_K} \xi_{K'}.
\end{equation}
\end{lemma}

Thanks to the equivalence result in Lemma \ref{lm:equivalenceConfusion},
the reliability and efficiency of the hybrid estimator $\xi$ follow immediately from
the corresponding results of the residual estimator $\eta$.


\begin{theorem}
    In the \textnormal{diffusion-dominated} regime,
    let $u$ be the exact solution of \eqref{eq:weakform} and $\uapprox$ be the finite element solution in \eqref{eq:feformDiff}.
    For the hybrid estimator $\xi$ defined in \eqref{eq:xiConfusion},
    there exists a constant $C_1$ such that 
    \[
        \xi_K \leq C_1 \left( \enorm{\uapprox-u}_{\omega_K} 
        + \sum_{K'\subseteq \omega_K} \Theta_{K'} \right), 
        \quad \forall\, K \in \meshth.
    \]
    Furthermore, under the monotonicity assumption \textnormal{\cite[Hypothesis 2.7]{bernardi2000}} of $\alpha$,
    there exists a positive constant $C_2$ such that
    \[
        \enorm{u-\uapprox}^2 \leq C_2 \left( \xi^2 + \sum_{K\in\meshth} \Theta_K^2 \right).
    \]
\end{theorem}
\begin{proof}
    The efficiency and reliability bounds are direct consequences of Lemma \ref{lm:equivalenceConfusion}, Theorem \ref{thm:etaeff} and \ref{thm:etarelia}.
\end{proof}

\begin{theorem}
\label{thm:xiConvDiff}
    In the \textnormal{convection/reaction-dominated} regime,
    let $u$ and $\uapprox$ be the solutions of \eqref{eq:weakform} and \eqref{eq:feformDiff} or \eqref{eq:feformConv}, respectively.
    The hybrid estimator $\xi$ defined in \eqref{eq:xiConfusion} satisfies
\[
    \enorm{u-\uapprox}^2 + \enorm{\vecta\cdot\grad(u-\uapprox)}_*^2 \leq C_1 
    \left( \xi^2+ \sum_{K\in\meshth}  \Theta_K^2  \right) 
\]
and 
\[
    \xi^2  \leq C_2 \left( \enorm{u-\uapprox}^2 + 
    \enorm{\vecta\cdot\grad(u-\uapprox)}_*^2 +  \sum_{K\in\meshth} \Theta_K^2 \right),
\]
where the $C_1$ and $C_2$ are positive constants.
\end{theorem}
\begin{proof}
    The theorem follows from Lemma \ref{lm:equivalenceConfusion} and Theorem \ref{thm:etaConvDiff}.
\end{proof}

For the singularly perturbed reaction-diffusion equation \eqref{eq:modelReaction},
in addition to the \emph{global} efficiency bound in Theorem \ref{thm:xiConvDiff},
$\xi_K$ satisfies a \emph{local} efficiency bound as a counterpart of Theorem \ref{thm:etaReaction}.
\begin{theorem}
    Let $u$ and $\uapprox$ be the respective exact and finite element solutions
    of the singularly perturbed reaction-diffusion equation in \eqref{eq:modelReaction}.
    The hybrid error indicator $\xi_K$ in \eqref{eq:xiConfusion}
    satisfies the following estimate
\[
    \xi_K^2 \leq C \left( \enormx{u-\uapprox}{\omega_K}^2 +
             \sum_{K'\subset \omega_K} \Theta_{K'}^2 \right),
\]
where $C$ is a positive constant.
\end{theorem}

\begin{remark}
    For the singularly perturbed reaction-diffusion equation \eqref{eq:modelReaction},
    in addition to the robust residual estimator in \cite{verf1998reaction},
    other types of estimators have been proposed over the years.
A general recovery-based estimator was proposed in \cite{caizhang2010}
based on projecting the numerical flux onto an $\hdiv$-conforming space. 
The global $L^2$-projection, however, may be regarded computationally expensive by some researchers.
A series of estimators based on flux equilibration
were explored in \cite{ainsworth1999reaction,ainsworth2011confusion,ainsworth2014reaction},
where the estimators yield guaranteed upper bounds of the true error.
The flux equilibration techniques in \cite{ainsworth2011confusion,ainsworth2014reaction}
require solving local least square problems associated with each vertex patch and only piecewise linear finite element approximations were discussed.
For diffusion equations, flux equilibration was discussed for higher order elements in \cite{equiflux}.
For convection-diffusion-reaction equations,
a flux equilibration procedure for linear elements was proposed in \cite{ainsworth2013confusion}, where the resulting estimator yields a guaranteed upper bound of the true error.
The results in \cite{ainsworth2013confusion} were derived under the assumption $\ddiv\vecta=0$
and the robustness was restricted to the special case of vanishing reaction $($i.e., $\beta=0)$.
\end{remark}

\section{Proof of Lemma \ref{lm:equivalenceConfusion}} 
\label{sec:proofConfusion}
In diffusion-dominated regime,
the proof of \eqref{eq:equivalenceConfusion} is the same as that of \cite[Theorem 4.2]{localL2}.
In convection/reaction-dominated regime,
the proof of \eqref{eq:equivalenceConfusion} is more complicated.
We first prove the upper bound of \eqref{eq:equivalenceConfusion} in Section \ref{sub:Upper bound}.
The lower bound of \eqref{eq:equivalenceConfusion} will be justified
in Section \ref{sub:lower12}.

\subsection{Upper bound} 
\label{sub:Upper bound}
The proof of the upper bound
is analogous to that in \cite[Theorem 4.2]{localL2},
which is essentially based on the proof of local efficiency of residual estimator
using properly chosen bubble functions.

Let $\psi_K$ be the standard element bubble function on $K$ and $\psi_e$ be modified face bubble function associated with $e$ defined in \cite{verf2005confusion}.
Unlike standard bubble functions that only depend on the geometry of the mesh, the modified face bubble function $\psi_e$ also involves $\epsilon$ and $\beta$ in order to fulfill the estimates in \eqref{eq:bubblebounds}.
It is known from \cite[Lemma 3.6]{verf2005confusion} that
\begin{equation}
\label{eq:bubblebounds}
\left\{\begin{alignedat}{1}
    \twonorm{\psi_K r_K}_K &\leq \twonorm{r_K}_K
    \leq C(\psi_K r_K, r_K)_K, \\[1.5mm]
    \twonorm{\grad (\psi_K r_K)}_K &\leq C \gammaK^{-1} h_K^{-1}\twonorm{r_K}_K,\\[1.5mm]
    \twonorm{j_e}_e^2 &\leq C (j_e,\psi_e j_e)_e,\\[1.5mm]
    \twonorm{\grad (\psi_e j_e)}_K &\leq C \gammae^{-1/2} h_e^{-1/2}\twonorm{j_e}_e,
    \quad K\subseteq \omega_e,\\[1.5mm]
    \text{and}\quad
    \twonorm{\psi_e j_e}_K &\leq C \gammae^{1/2} h_e^{1/2}\twonorm{j_e}_e, 
    \quad K\subseteq \omega_e.
\end{alignedat}\right.
\end{equation}

\begin{proof}[Proof of the upper bound in \eqref{eq:equivalenceConfusion}]
With the help of \eqref{eq:bubblebounds}, \eqref{eq:rhatK}, integration by parts, 
the fact that $\psi_K$ vanishes on $\partial K$,
and the Cauchy-Schwarz inequality, 
we deduce that
\begin{equation*}
\begin{aligned}
    C_1\twonorm{r_K}_K^2
    &\leq \left( r_K, \psi_K r_K  \right)_K 
            = (\ddiv(\fluxmesh-\fluxh), \psi_K r_K )_K + (\rhatK, \psi_K r_K)_K  \\[2mm]
    &= (\fluxh-\fluxmesh,\grad (\psi_K r_K))_K +  (\rhatK, \psi_K r_K)_K \\[2mm]
    &\leq C_2 \twonorm{r_K}_K \left( \gammaK^{-1} h_K^{-1}\twonorm{\fluxmesh-\fluxh}_K + \twonorm{\rhatK}_K \right),
\end{aligned}
\end{equation*}
which implies
\begin{equation}
\label{eq:resxi}
    \gammaK h_K\epsilon^{-1/2}\twonorm{r_K}_K
    \leq C \big( \epsilon^{-1/2}\twonorm{\fluxmesh-\fluxh}_K + \gammaK h_K\epsilon^{-1/2}\twonorm{\rhatK}_K \big)\leq C\xi_K.
\end{equation}

To estimate $\twonorm{j_e}_e$,
it follows from \eqref{eq:bubblebounds}, integration by parts, the Cauchy-Schwarz and the triangle inequalities,
and \eqref{eq:resxi} that
\begin{equation*}
\begin{aligned}
    \twonorm{j_e}^2 &\leq C (j_e,\psi_e j_e)_e \\
    &= C \sum_{K\subseteq \omega_e} \left( (\fluxh-\fluxmesh,\grad (\psi_e j_e))_K + (\ddiv(\fluxh-\fluxmesh),\psi_e j_e)_K  \right) \\
    &\leq C \frac{\epsilon^{1/2}\twonorm{j_e}_e}{\gammae^{1/2} h_e^{1/2}} 
        \sum_{K\subseteq \omega_e} \left( \epsilon^{-1/2}\twonorm{\fluxmesh-\fluxh}_K 
          + \gammae h_e\epsilon^{-1/2}(\twonorm{\rhatK}_K + \twonorm{r_K}_K)
             \right) \\
  &\leq C \frac{\epsilon^{1/2}\twonorm{j_e}_e}{\gammae^{1/2} h_e^{1/2}}\sum_{K\subseteq \omega_e} \xi_K.
\end{aligned}
\end{equation*}
Now, the upper bound in \eqref{eq:equivalenceConfusion} is a direct consequence of
\eqref{eq:resxi} and the definition of $\eta_{K}$ in \eqref{eq:etaConfusion}.
\end{proof}


\subsection{Lower bound}
\label{sub:lower12}
In this section, we prove the lower bound in \eqref{eq:equivalenceConfusion}
in the convection/reaction-dominated regime. 
That is, for any $K\in\meshth$, there exists a positive constant $C$
independent of $h_K$, $\epsilon$, and $\beta$ such that
\begin{equation}
\label{eq:lower12}
    \xi_K\leq C \eta_K.
\end{equation}
This is proceeded in two cases:
(i) $R_k\leq\kpaneg$ and (ii) $R_k>\kpaneg$.

\begin{proof}[Proof in \textnormal{Case (i)}]
In this case, 
the inequality $h_K>R_K$ and \eqref{eq:shaperegular} give 
\begin{equation*}
    1\geq \gammaK = h_K^{-1} \min\{ h_K, \epsilon^{1/2}\beta^{-1/2}\}
    \geq  h_K^{-1} \min\{ h_K, \frac{h_K}{C_0} \} = \frac{1}{C_0},
\end{equation*}
which implies the weight $\gammaK=O(1)$ is independent of $h_k$, $\epsilon$, and $\beta$.
Hence, to show the validity of \eqref{eq:lower12}, it suffices to prove that 
\begin{equation}
\label{eq:lower1proof}
    \twonorm{\fluxmesh-\fluxh}_K^2 + h_K^2\twonorm{\rhatK}_K^2
    \leq C  \left( h_K^2 \twonorm{r_K}_K^2 + \sum_{e\in\edgek} h_e\twonorm{j_e}_e^2 \right).
\end{equation}
To this end, note first that $\ddiv\fluxh|_K \in P_{k-1}(K)$
and that the recovered flux defined in \eqref{eq:fluxDiffConv} satisfies
\begin{equation}
\label{eq:divdiff1} 
    \ddiv\left(\fluxmesh-\fluxh\right)= \Pi_K^{k-1} r_K + J_K,
\end{equation}
which, together with \cite[Lemma 4.1]{localL2} and the triangle inequality,
implies 
\begin{equation}
\label{eq:fluxbound1}
\begin{aligned}
    \twonorm{\fluxmesh-\fluxh}_K 
    &\leq C \left( h_K\twonorm{\ddiv(\fluxmesh-\fluxh)}_K + \sum_{e\in\edgek} h_e^{1/2}\twonorm{j_e}_e  \right) \\
    &\leq C  \left( h_K\twonorm{r_K}_K + h_K\twonorm{J_K}_K + \sum_{e\in\edgek}h_e^{1/2}\twonorm{j_e}_e \right).
\end{aligned}
\end{equation}
To bound the modified element residual $\rhatK$ in \eqref{eq:rhatK},
by \eqref{eq:divdiff1} we have
\begin{equation*}
    \rhatK = r_K - \ddiv(\fluxmesh-\fluxh) = \left(I-\Pi_K^{k-1}\right)r_K - J_K,
\end{equation*}
which yields
\begin{equation}
\label{eq:rhatKbound1}
    \twonorm{\rhatK}_K\leq \twonorm{r_K}_K+\twonorm{J_K}_K.
\end{equation}
Now, \eqref{eq:lower1proof} is a direct consequence of \eqref{eq:fluxbound1}, \eqref{eq:rhatKbound1}, and the following bound
\[
  \twonorm{J_K}_K \leq C \left( \twonorm{r_K}_K+\sum_{e\in\edgek}h_e^{-1/2}\twonorm{j_e}_e \right),
\]
which follows from the divergence theorem and the triangle and the Cauchy-Schwarz inequalities that
\[
\begin{aligned}
    |J_K| &= |K|^{-1} \abs{\int_{\partial K} (\ghatK-\fluxh|_K\cdot\normalvect) ds - \int_K r_K dx} \\
    &\leq |K|^{-1} \left( \sum_{e\in\edgek}h_e^{1/2}\twonorm{j_e}_e+|K|^{-1/2}\twonorm{r_K}_K \right).
\end{aligned}
\]
This completes the proof of \eqref{eq:lower1proof} and, hence, \eqref{eq:lower12}.
\end{proof}

\begin{proof}[Proof in \textnormal{Case (ii)}]
When $R_K > \kpaneg$,
the fact that $h_S > R_S$ implies
\begin{equation}
\label{eq:gammaconvX}
    \gamma_{_S} = \epsilon^{1/2}h_S^{-1}\beta^{-1/2}, \quad S = K \text{ or } S\in\edgek.
\end{equation}
To prove \eqref{eq:lower12}, it suffices to show that 
\begin{equation}
\label{eq:lower2proof}
    \epsilon^{-1}\twonorm{\fluxmesh-\fluxh}_K^2 + \beta^{-1}\twonorm{\rhatK}_K^2
    \leq C  \left( \beta^{-1} \twonorm{r_K}_K^2 + \sum_{e\in\edgek} \epsilon^{-1/2}\beta^{-1/2}\twonorm{j_e}_e^2 \right).
\end{equation}
To this end, we first estimate $\twonorm{\fluxmesh-\fluxh}_K$.
Note that 
\[
    |\fluxmesh-\fluxh| = \begin{cases}
        0, &\text{in } K_{\Delta},\\
        (1-\lambda_{K,e})(1-\epsilon^{-1/2}\beta^{1/2}y)|j_e(\vect{x}_{Q_e}+x\vect{t})|, &\text{in } Q_e\\
    \end{cases}
\]
for all $e\in\edgek$, where $x,y$ are local coordinates in $Q_e$,
and 
\begin{equation}
\label{eq:diverrorJw}
    \ddiv(\fluxmesh-\fluxh) = J_{\omega}\quad \text{in}\; \omega = T_e, S_e.
\end{equation}
Define at this moment $Q_K=\cup_{e\in\edgek}Q_e$. A straightforward calculation gives
\begin{equation}
\label{eq:fluxbound2Q}
    \epsilon^{-1}\twonorm{\fluxmesh-\fluxh}_{K_{\Delta}\cup Q_K}^2 
    =\epsilon^{-1}\sum_{e\in\edgek}\twonorm{\fluxmesh-\fluxh}_{Q_e}^2 
    \leq \frac{\epsilon^{-1/2}\beta^{-1/2}}{3}\sum_{e\in\edgek}\twonorm{j_e}_e^2.
\end{equation}
The estimate of $\twonorm{\fluxmesh-\fluxh}_{\omega}$ for $\omega=T_e$ and $S_e$ is analogous to Case (i).
It follows from \eqref{eq:diverrorJw}, the fact that $\text{diam}(\omega)=O(\kpaneg)$, and the Cauchy-Schwarz inequality that
\begin{equation}
\label{eq:Jwbound} 
    \twonorm{\ddiv(\fluxmesh-\fluxh)}_{\omega} = \twonorm{J_{\omega}}_{\omega} 
    \leq C \epsilon^{-1/4}\beta^{1/4}\twonorm{j_e}_e,
\end{equation}
which, together with \cite[Lemma 4.1]{localL2}, implies
\[
\begin{aligned}
    \epsilon^{-1/2}\twonorm{\fluxmesh-\fluxh}_{\omega} 
     &\leq C \left( \beta^{-1/2}\twonorm{J_{\omega}}_{\omega} + \epsilon^{-1/4}\beta^{-1/4}\twonorm{j_e}_e \right) 
    \leq C \epsilon^{-1/4}\beta^{-1/4}\twonorm{j_e}_e
\end{aligned}
\]
for all $e\in\edgek$.
Combining with \eqref{eq:fluxbound2Q} gives
\begin{equation}
\label{eq:fluxbound2}
    \epsilon^{-1}\twonorm{\fluxmesh-\fluxh}_K^2 \leq C \sum_{e\in\edgek} \epsilon^{-1/2}\beta^{-1/2}\twonorm{j_e}_e^2.
\end{equation}
It remains to estimate $\beta^{-1}\twonorm{\rhatK}_K^2$. By \eqref{eq:rhatK} and the triangle inequality, we have 
\begin{equation}
\label{eq:rhatKleqrkdiv}
    \twonorm{\rhatK}_K \leq \twonorm{r_K}_K + \twonorm{\ddiv(\fluxmesh-\fluxh)}_K.
\end{equation}
It follows from the definition of $\fluxmesh$ that
\[
    \abs{\ddiv(\fluxmesh-\fluxh)} = \begin{cases}
        0, &\text{in } K_{\Delta}, \\
        (1-\lambda_e)\epsilon^{-1/2}\beta^{1/2}|j_e(\vect{x}_{Q_e}+x\vect{t})|, &\text{in } Q_e, \\
        |J_{\omega}|, &\text{in } \omega = T_e, S_e
    \end{cases}
\]
for all $e\in\edgek$.
A straightforward calculation yields
\begin{equation*}
\sum_{e\in\edgek}\twonorm{\ddiv(\fluxmesh-\fluxh)}_{Q_e} \leq \epsilon^{-1/4}\beta^{1/4}\sum_{e\in\edgek}\twonorm{j_e}_e,
\end{equation*}
which, together with \eqref{eq:Jwbound},
implies
\begin{equation}
\label{eq:diverrorleqje}
    \twonorm{\ddiv(\fluxmesh-\fluxh)}_K\leq  C \epsilon^{-1/4}\beta^{1/4}\sum_{e\in\edgek}\twonorm{j_e}_e.
\end{equation}
\eqref{eq:fluxbound2} -- \eqref{eq:diverrorleqje} 
imply \eqref{eq:lower2proof}
and the proof is complete.

\end{proof}




\section{Numerical Experiments} 
\label{sec:numericalConfusion}
In numerical experiments, we consider the singularly perturbed reaction-diffusion problem in \eqref{eq:modelReaction}.
Here the domain is chosen as $\Omega = [-1,1]^2$.
The initial mesh in the adaptive mesh refinement consists of $4\times 4$ congruent squares, 
each of which is partitioned into two triangles connecting bottom-left and top-right corners.
As in \cite{dorfler1996, AFEM2002conv, localL2}, we use D{\"o}rfler's marking strategy \cite{dorfler1996} with $\theta_D=0.5$ (cf. \cite[Eq.(5.1)]{localL2}).
The newest-vertex bisection \cite{purdue1972} is used in the refinement.
For the finite element discretization, 
$P_1$ conforming element is used in all examples.
The exact error is denoted by $e = u-\uapprox$.
``DOFs" denotes the degrees of freedom and
``eff-ind" denotes the effectivity index, namely, either $\eta/\enorm{e}$ or $\xi/\enorm{e}$.

\begin{remark}
The performance of the hybrid estimator for diffusion-dominated problems can be seen from the numerical results in \cite{localL2}.
\end{remark}

\subsection{Test Problem 1} 
We first consider an example as in \cite[Example 1]{ainsworth2011confusion},
where the solution is smooth but, as pointed out in \cite{ainsworth2011confusion}, non-robust estimators do not perform well .
The exact solution is chosen as
\[
    u(x,y) = \cos(\pi x /2)\cos(\pi y/2)/(1+\epsilon\pi^2/2)
\]
and the data is $f(x,y) = \cos(\pi x /2)\cos(\pi y/2)$.
The aim of this example is to numerically show that the hybrid estimator $\xi$ is more accurate
than the residual estimator $\eta$
and to demonstrate that $\xi$ is less sensitive with respect to the variation of $\epsilon$.

\subsubsection{Effectivity with respect to $\epsilon$} 
\label{ssub:Effectivity}
On a fixed uniform mesh composed of $200$ isosceles right triangles,
we vary $\epsilon$ and investigate the change of effectivity index for each estimator.
Numerical results are collected in Table \ref{tab:effectivity} for different choices of $\epsilon$.
It is easily seen from Table \ref{tab:effectivity} that
the hybrid estimator $\xi$ is more accurate and than the residual estimator $\eta$.
When $\epsilon$ changes gradually from $10^{-5}$ to $100$,
the effectivity indices for $\xi$ remain close to $1$,
while the effectivity indices for $\eta$ increase significantly from $0.66$ to $5.56$, 
by a factor of $8.4$.
This implies that, compared to the explicit residual estimator $\eta$, the hybrid estimator $\xi$ is less sensitive to the size of reaction.

\begin{table}
\caption{Example 1 - Effectivity indices for different $\epsilon$ on a fixed mesh}
\label{tab:effectivity}
\begin{center}
\begin{tabular}{|c|c|c|c|c|c|c|c|c|c|c|c|c|}
\hline  
$\epsilon$ & 1E-5 & 1E-4 & 5E-4 & 1E-3 & 5E-3 & 1E-2 & 5E-2 & 1E-1 & 1 & 10 & 100 \\
\hline 
$\eta$     & 0.66 & 0.66 & 0.93 & 1.21 & 2.22 & 2.81 & 4.83 & 5.58 & 5.57 & 5.56 & 5.56 \\
\hline 
$\xi$      & 0.80 & 0.84 & 1.09 & 1.35 & 1.20 & 1.28 & 1.38 & 1.38 & 1.36 & 1.36 & 1.36 \\
\hline 
\end{tabular}
\end{center}
\end{table}

\subsubsection{Effectivity during adaptive mesh refinement} 
\label{ssub:Effectivity during adaptive mesh refinement}
In this section,
for each fixed $\epsilon$, 
we investigate the effectivity of each estimator during adaptive mesh refinement.

The numerical results for the residual estimator $\eta$
are shown in Figure \ref{fig:res1e-43} -- \ref{fig:res1e-21} for different $\epsilon$.
(The result for $\epsilon=1$ or $10$ is quite similar to $\epsilon=10^{-1}$ and is thus not shown here.) 
It is obviously seen that the effectivity of $\eta$ strongly depends on the relation between $\epsilon$ and mesh-size.
Hence even though theoretically the constants in the a posteriori error estimates for $\eta$ are independent of $\epsilon$,
the practical performance does display a significant difference for different choices of $\epsilon$.
This is because the theoretical results can only provide 
lower and upper bounds of $\eta/\enorm{e}$, i.e., an interval that $\eta/\enorm{e}$ lies in,
and the exact value may vary in such a (possibly large) interval. 
In this case, the size of this interval could be quite large even though it is independent of $\epsilon$.

The numerical results for the hybrid estimator $\xi$
are shown in Figure \ref{fig:RT1e-43} -- \ref{fig:RT1e-21} for different $\epsilon$.
Unlike $\eta$, numerical results indicate that
$\epsilon$ does not have much influence on the effectivity of $\xi$ during the adaptive mesh refinement. 
Therefore, 
we see that the hybrid estimator is indeed less sensitive
than the residual estimator with respect to $\epsilon$.

\begin{figure}[htbp]
\centering
\begin{minipage}[t]{0.5\textwidth}
\centering
\includegraphics[width=.9\textwidth]{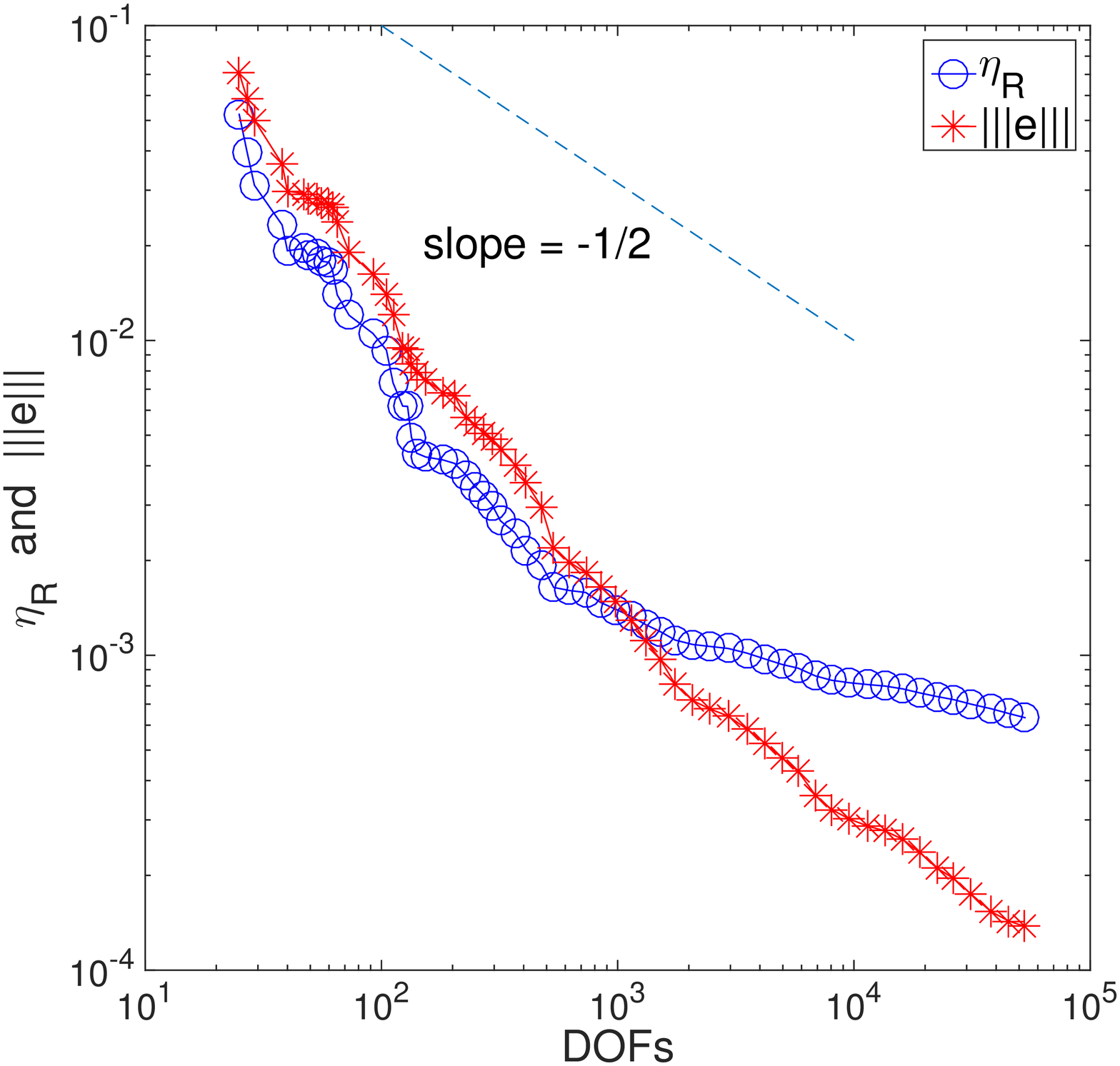}
\end{minipage}\hfill
\begin{minipage}[t]{0.5\textwidth}
\centering
\includegraphics[width=.9\textwidth]{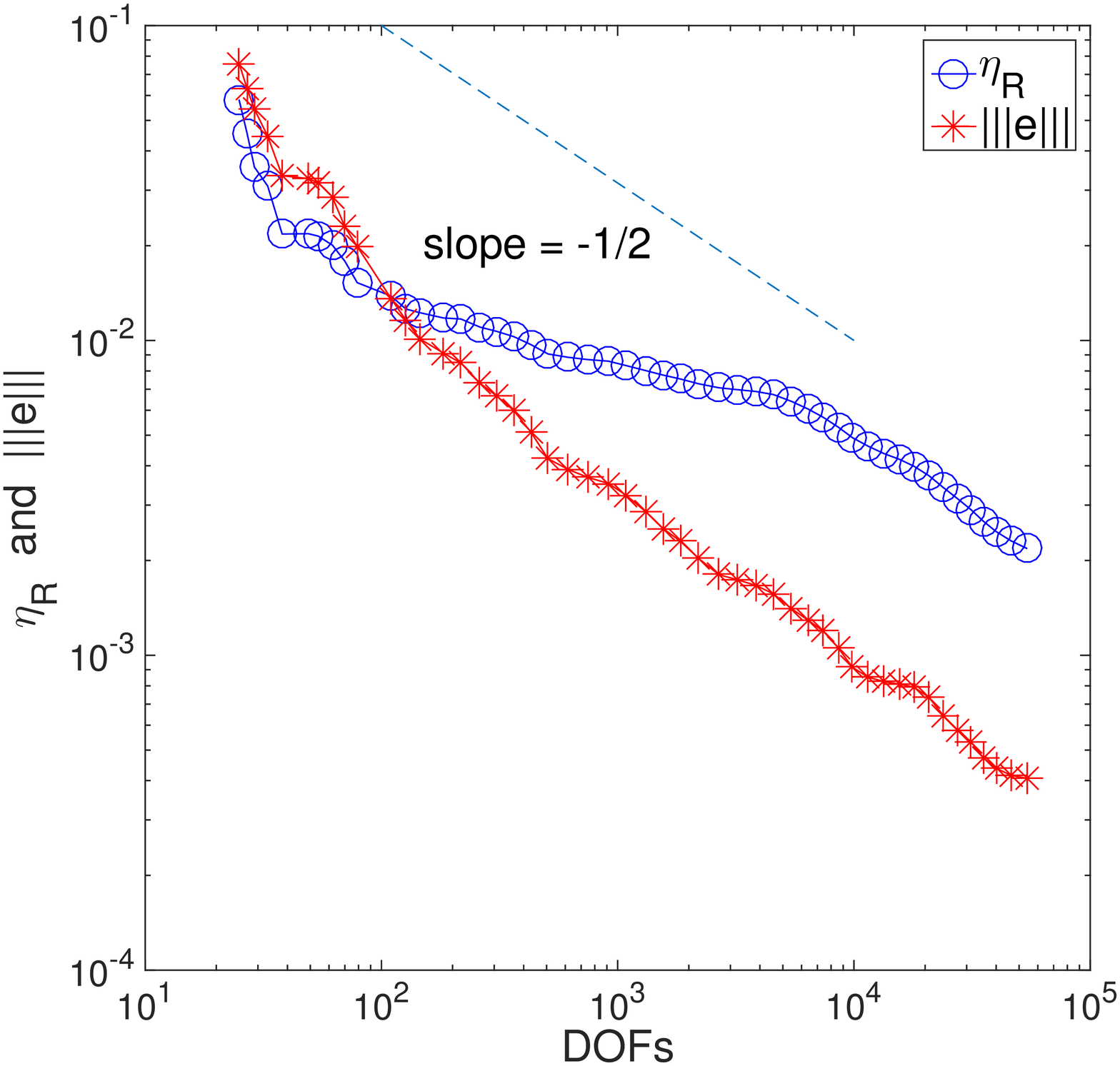}
\end{minipage}
\caption{Example 1 - $\eta$ and $\enorm{e}$ with $\epsilon=10^{-4}$(left), $\epsilon=10^{-3}$(right)}
\label{fig:res1e-43}
\end{figure}

\begin{figure}[htbp]
\centering
\begin{minipage}[t]{0.5\textwidth}
\centering
\includegraphics[width=.9\textwidth]{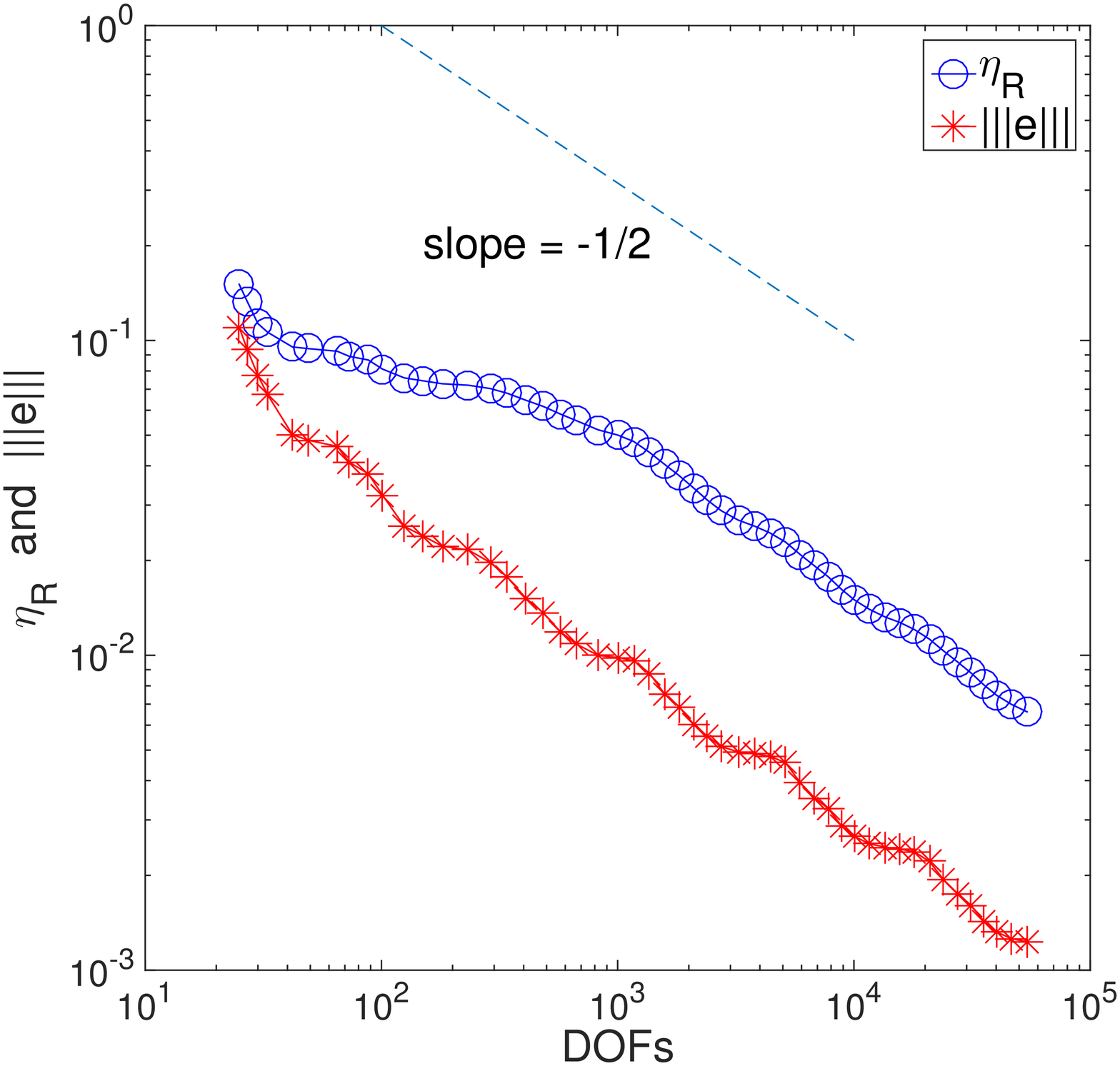}
\end{minipage}\hfill
\begin{minipage}[t]{0.5\textwidth}
\centering
\includegraphics[width=.9\textwidth]{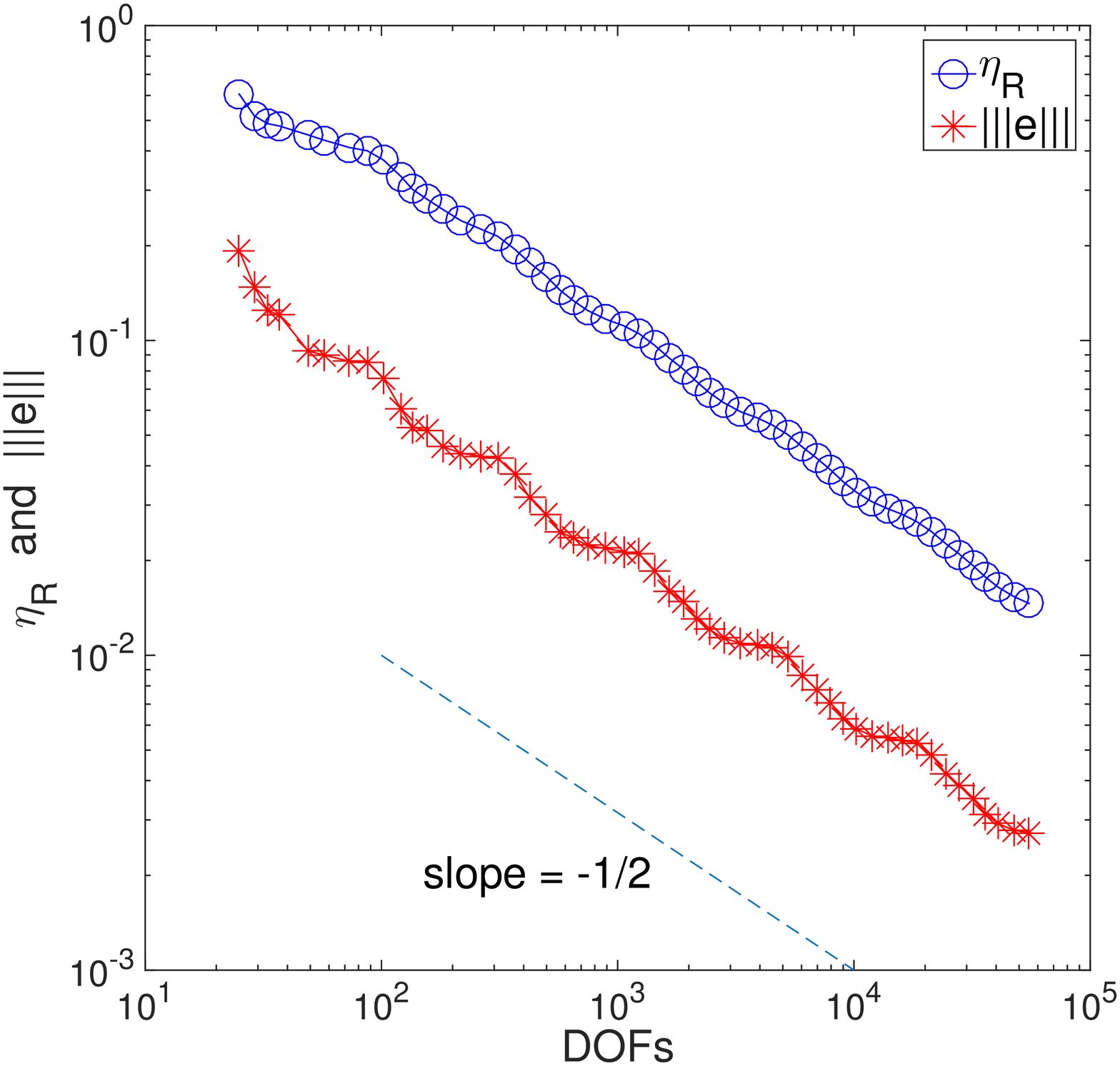}
\end{minipage}
\caption{Example 1 - $\eta$ and $\enorm{e}$ with $\epsilon=10^{-2}$(left), $\epsilon=10^{-1}$(right)}
\label{fig:res1e-21}
\end{figure}


%
%
%
%
%

\begin{figure}[htbp]
\centering
\begin{minipage}[t]{0.5\textwidth}
\centering
\includegraphics[width=.9\textwidth]{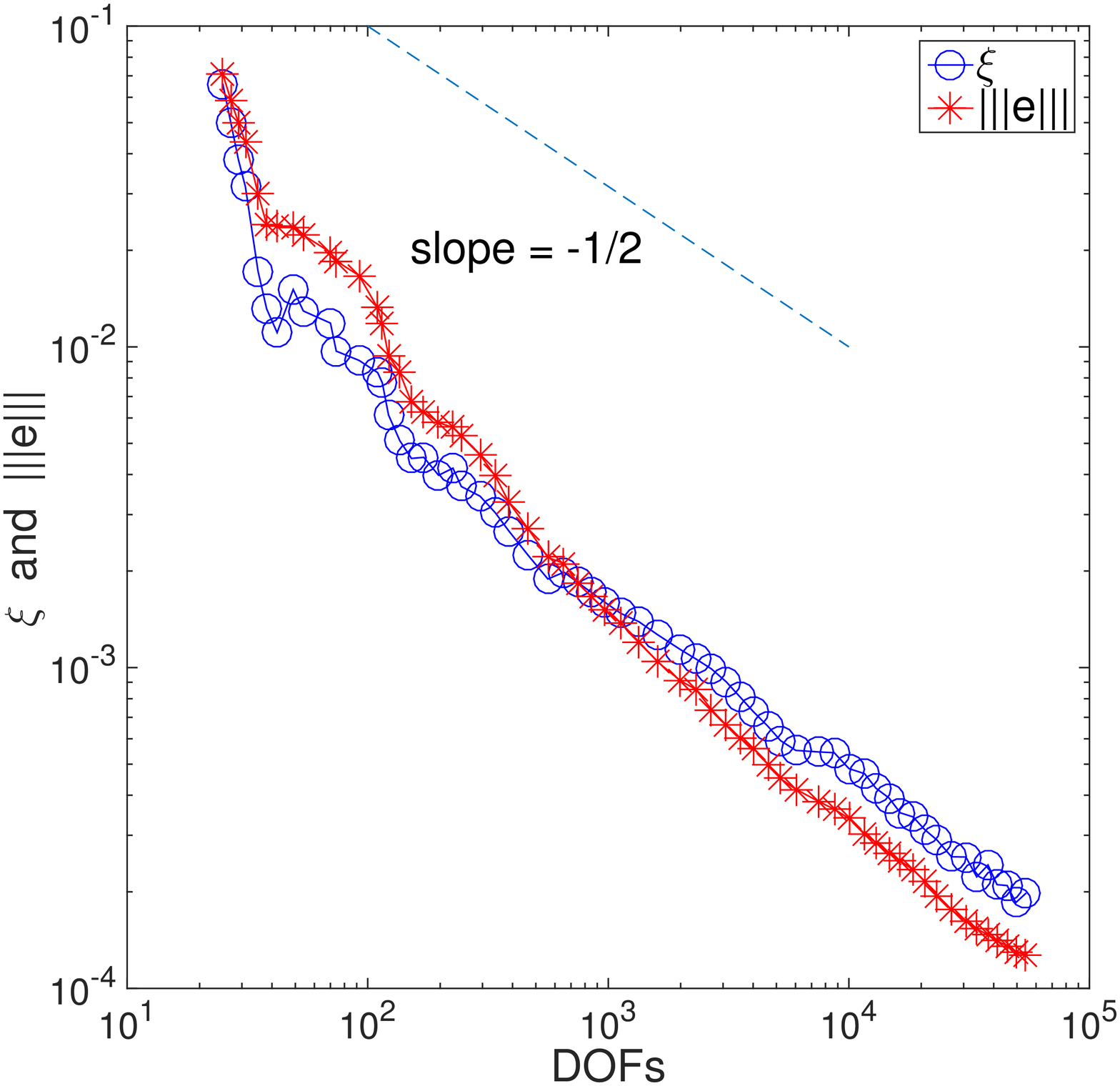}
\end{minipage}\hfill
\begin{minipage}[t]{0.5\textwidth}
\centering
\includegraphics[width=.9\textwidth]{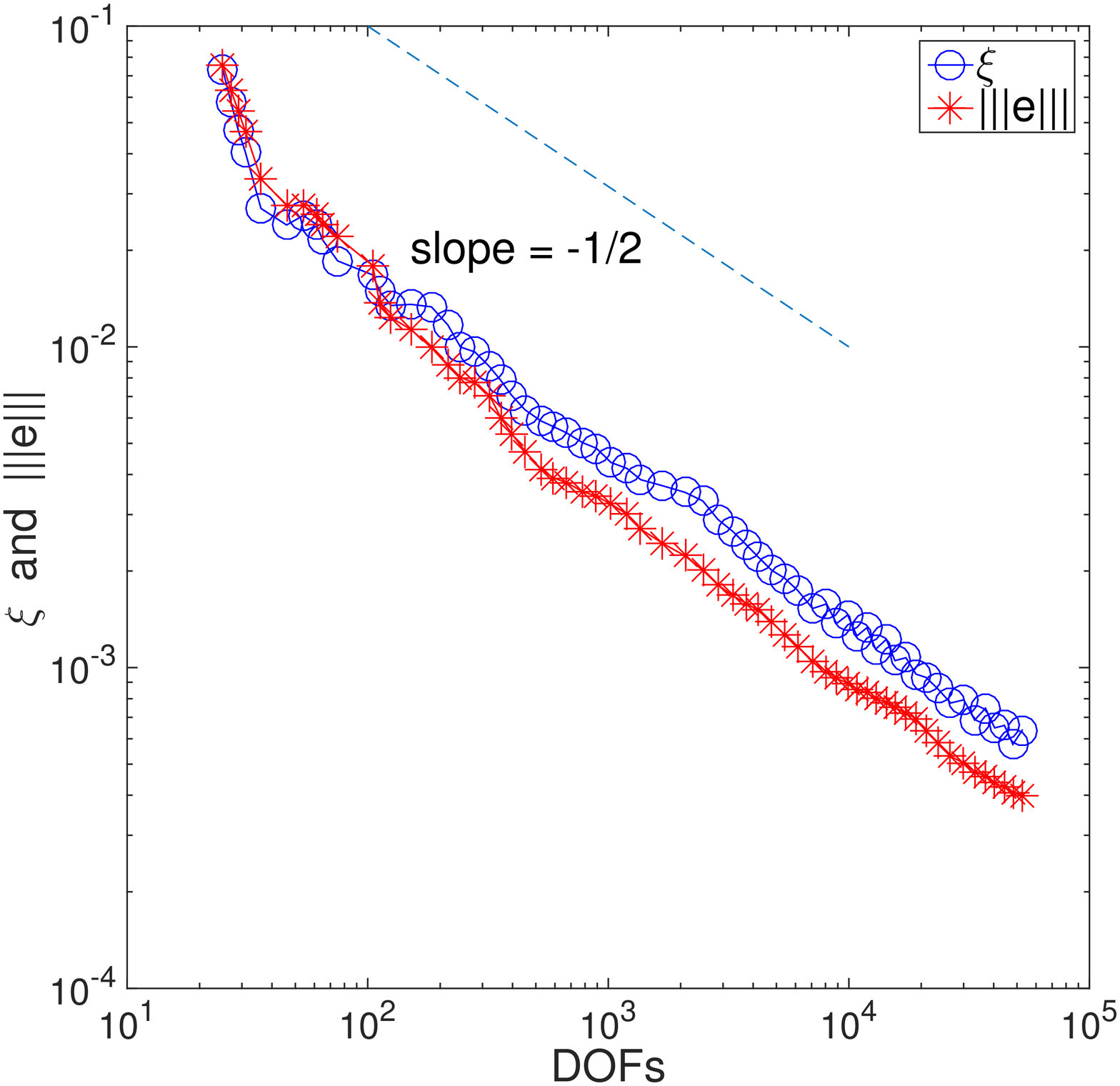}
\end{minipage}
\caption{Example 1 - $\xi$ and $\enorm{e}$ with $\epsilon=10^{-4}$(left), $\epsilon=10^{-3}$(right)}
\label{fig:RT1e-43}
\end{figure}

\begin{figure}[htbp]
\centering
\begin{minipage}[t]{0.5\textwidth}
\centering
\includegraphics[width=.9\textwidth]{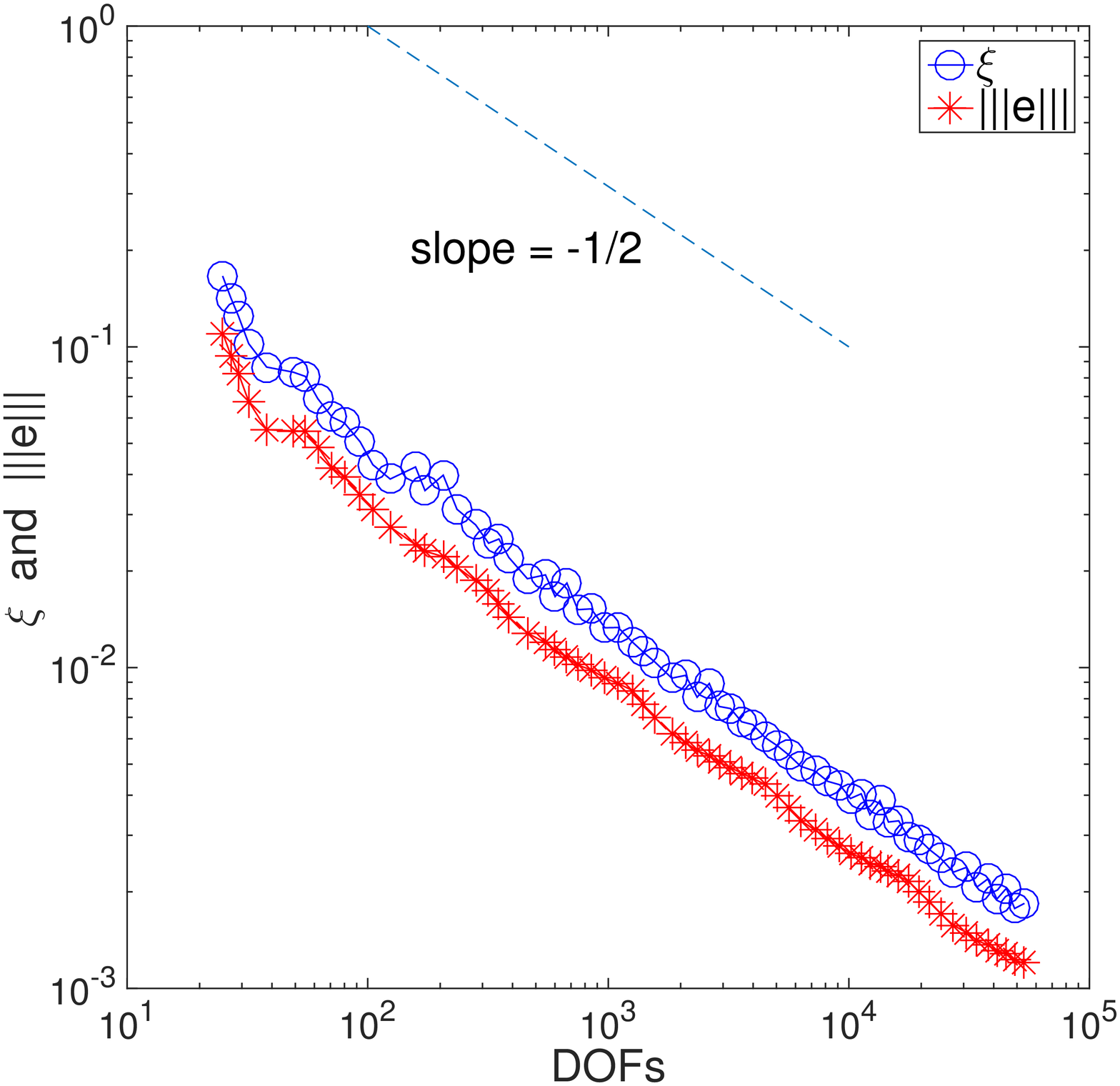}
\end{minipage}\hfill
\begin{minipage}[t]{0.5\textwidth}
\centering
\includegraphics[width=.9\textwidth]{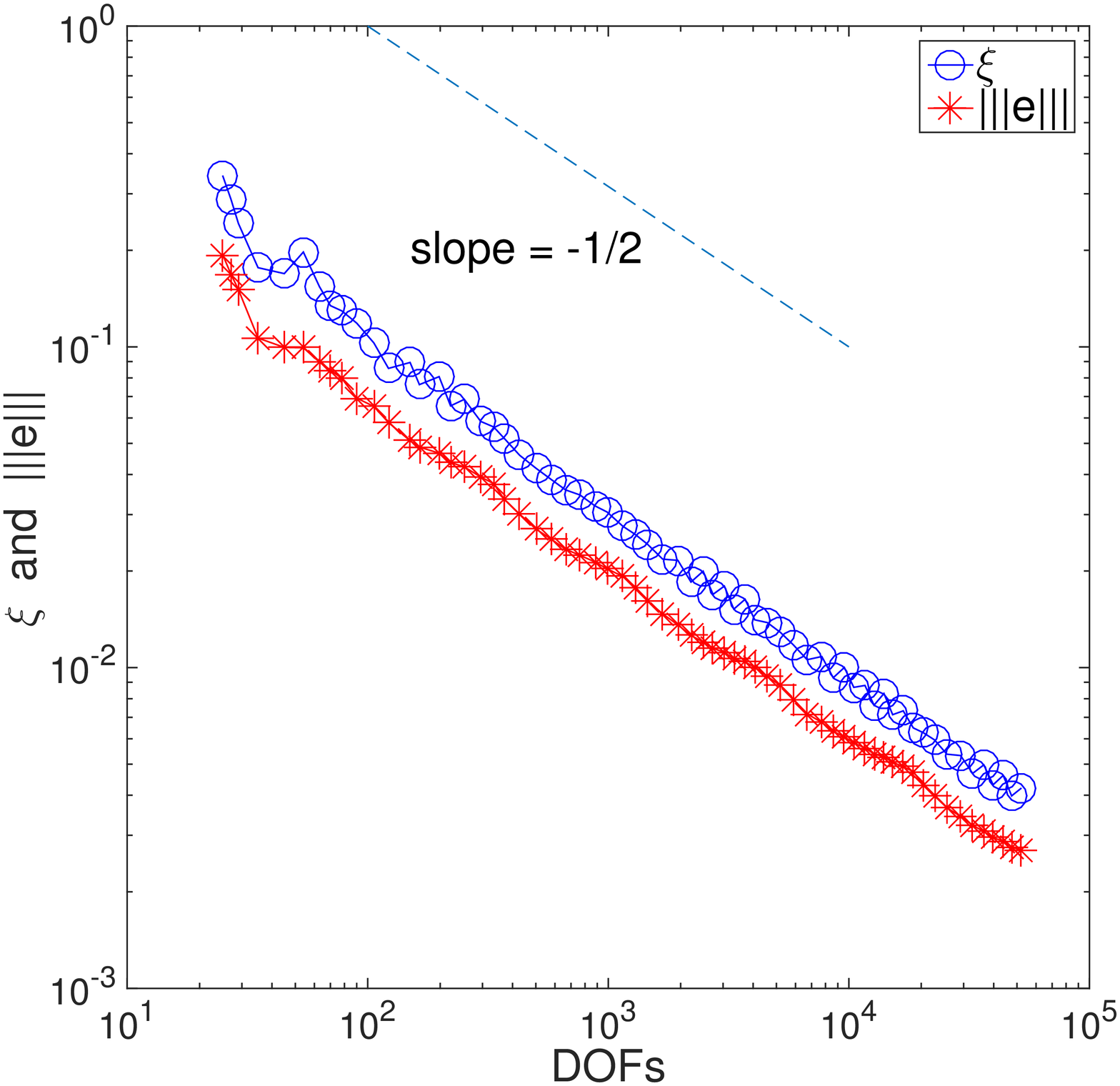}
\end{minipage}
\caption{Example 1 - $\xi$ and $\enorm{e}$ with $\epsilon=10^{-2}$(left), $\epsilon=10^{-1}$(right)}
\label{fig:RT1e-21}
\end{figure}



\subsection{Test Problem 2} 
Test problem 2 has $f=0$ and the exact solution 
\[
    u(x,y) = e^{-(x+1)/\sqrt{\epsilon}} + e^{-(y+1)/\sqrt{\epsilon}}
    \quad \text{with} \quad \epsilon = 10^{-4},
\]
which displays boundary layers along $x=-1$ and $y=-1$.

We perform adaptive mesh refinement
with stopping criterion
    $\enorm{e} \leq 0.1 \enorm{u}$
and compare the results obtained using $\eta$ and $\xi$.

The mesh generated by $\xi$ is shown in Figure \ref{fig:Ex23mesh} (left),
where major refinements are along the boundary layers.
From Table \ref{tab:Ex2TOL01},
it is easy to see that the residual estimator is less accurate than the hybrid estimator.
Figure \ref{fig:Ex2errorTOL01} again shows that
the effectivity index of $\eta$ varies more widely than that of $\xi$
during the mesh refinement procedure.
The robustness of $\xi$ as well as $\eta$ can be seen from Figure \ref{fig:Ex2errorTOL01} as 
the optimal error decay rate is observed on very coarse meshes (DOFs$\approx$100).
\begin{figure}[htbp]
\centering
\begin{minipage}[t]{0.5\textwidth}
\centering
\includegraphics[width=.7\textwidth]{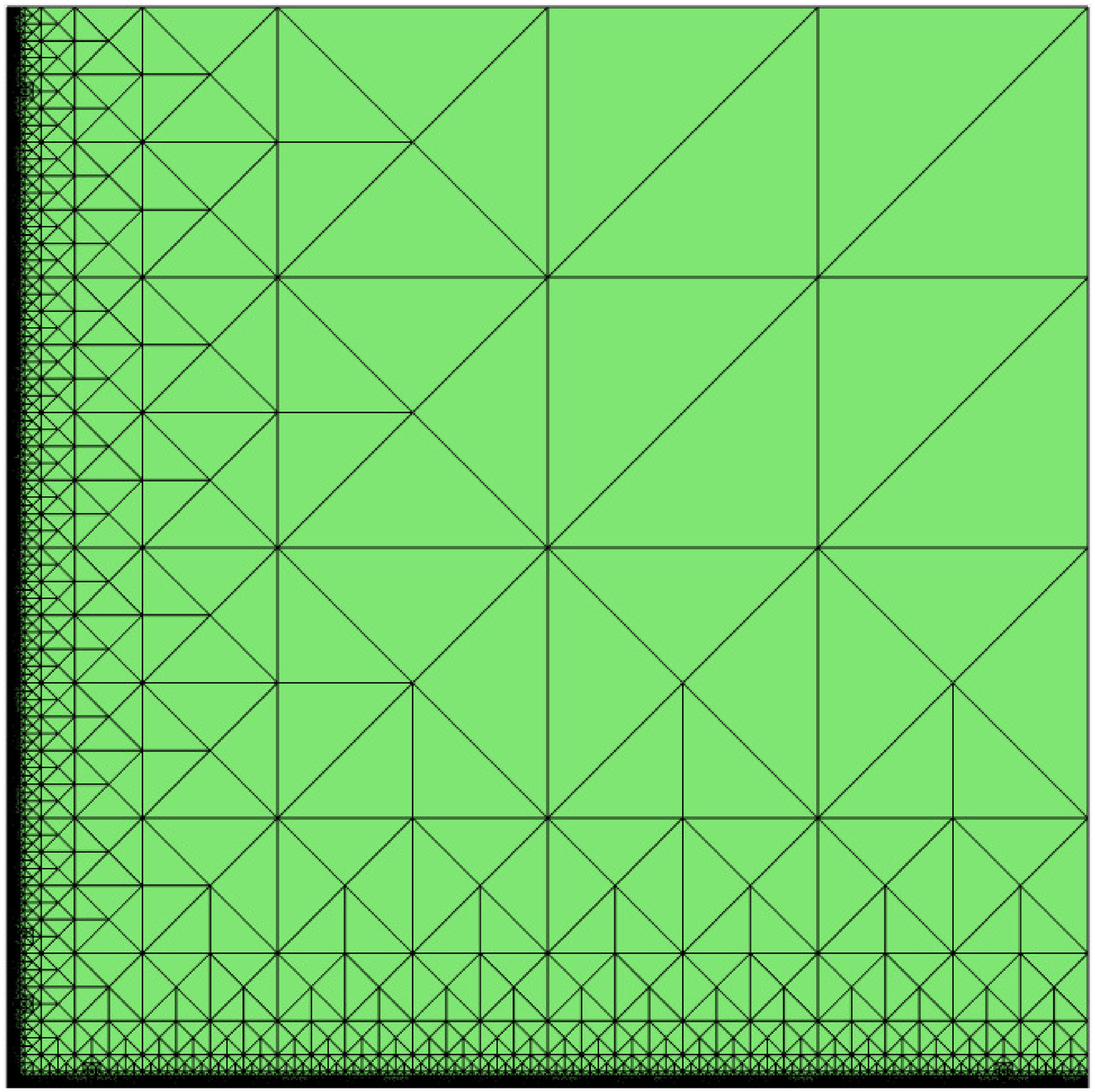}
\end{minipage}\hfill
\begin{minipage}[t]{0.5\textwidth}
\centering
\includegraphics[width=.7\textwidth]{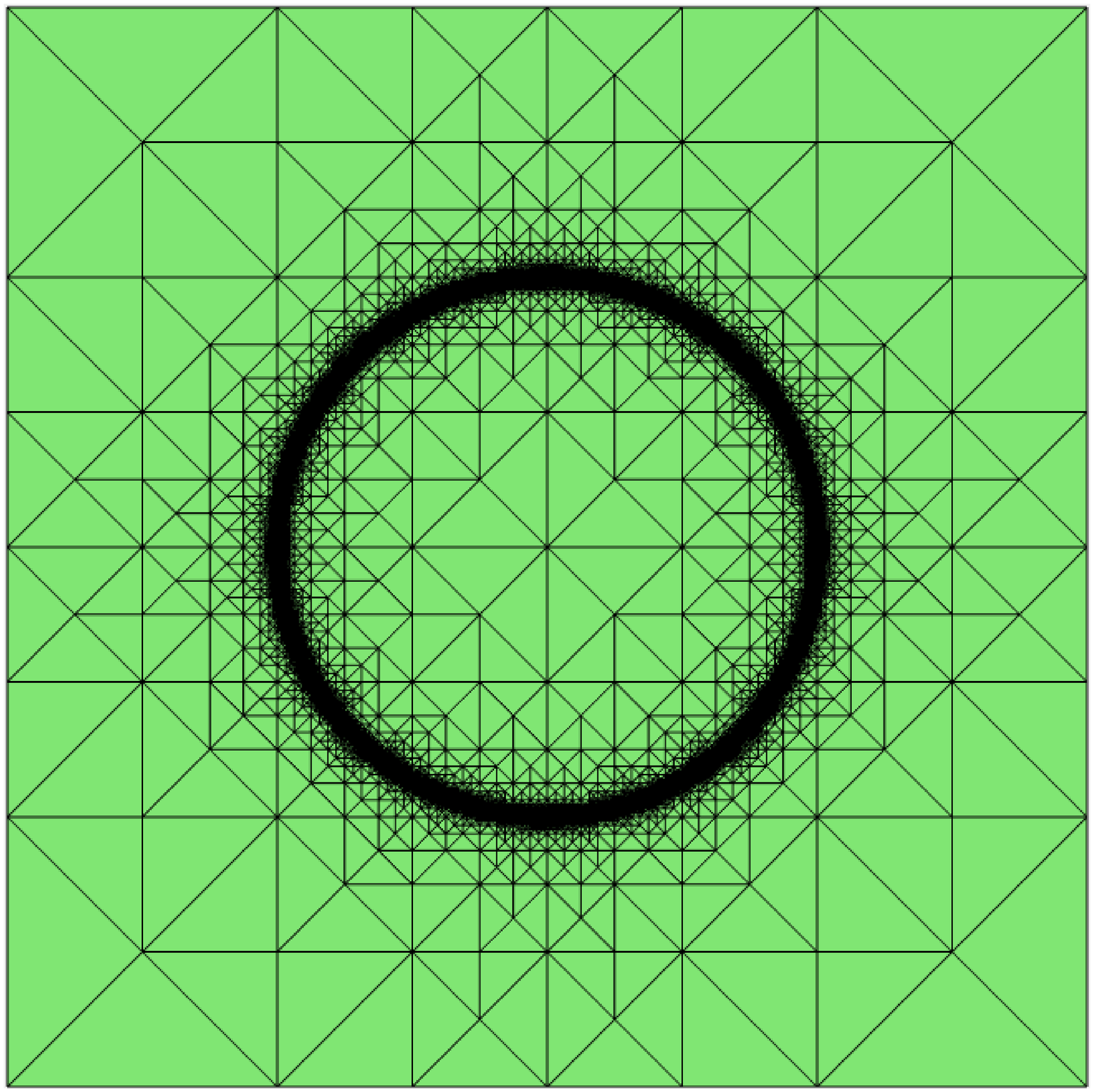}
\end{minipage}
\caption{Meshes generated using $\xi$: Example 2(left) and Example 3(right)}
\label{fig:Ex23mesh}
\end{figure}

\begin{figure}[htbp]
\centering
\begin{minipage}[t]{0.5\textwidth}
\centering
\includegraphics[width=.9\textwidth]{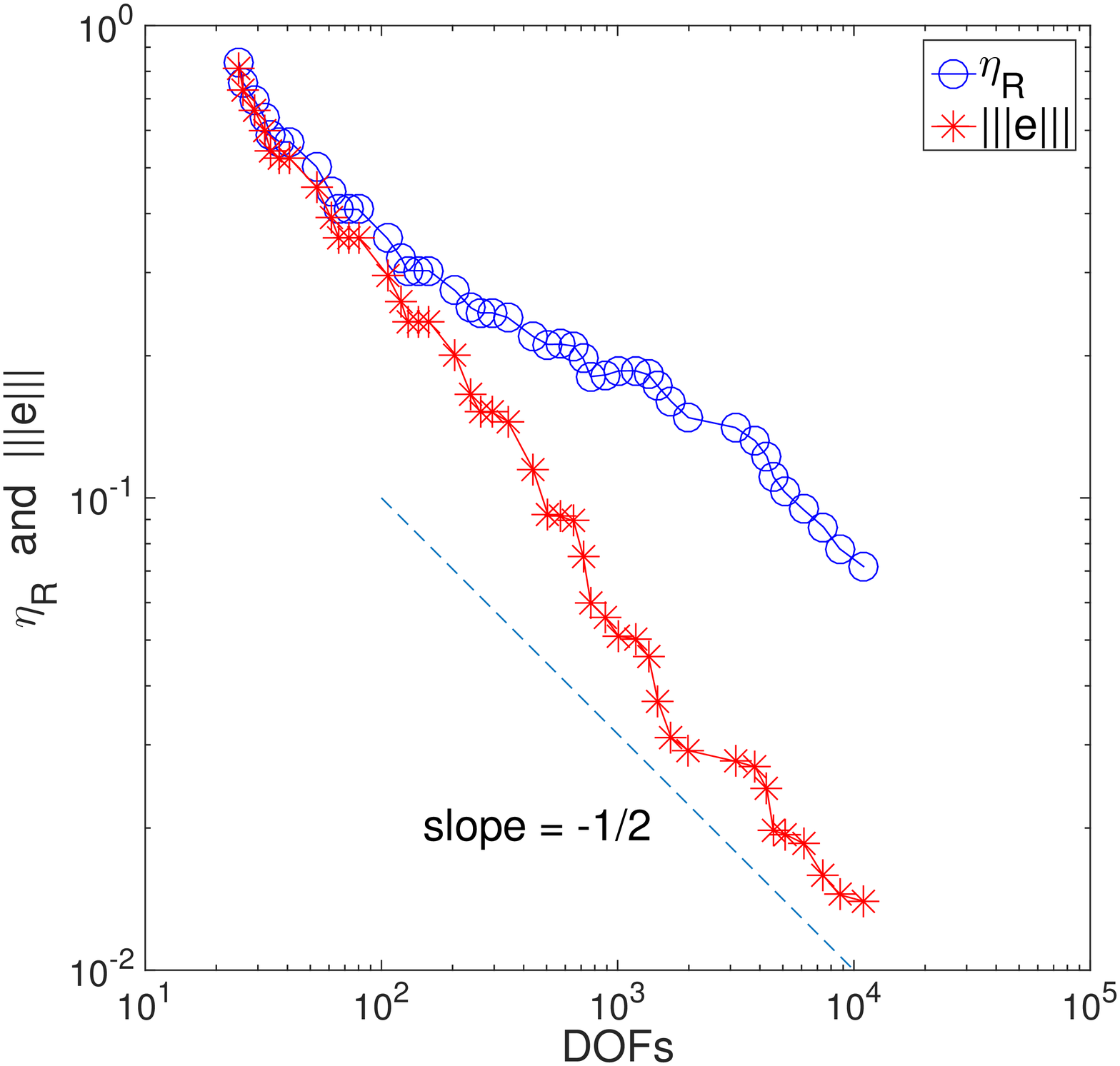}
\end{minipage}\hfill
\begin{minipage}[t]{0.5\textwidth}
\centering
\includegraphics[width=.9\textwidth]{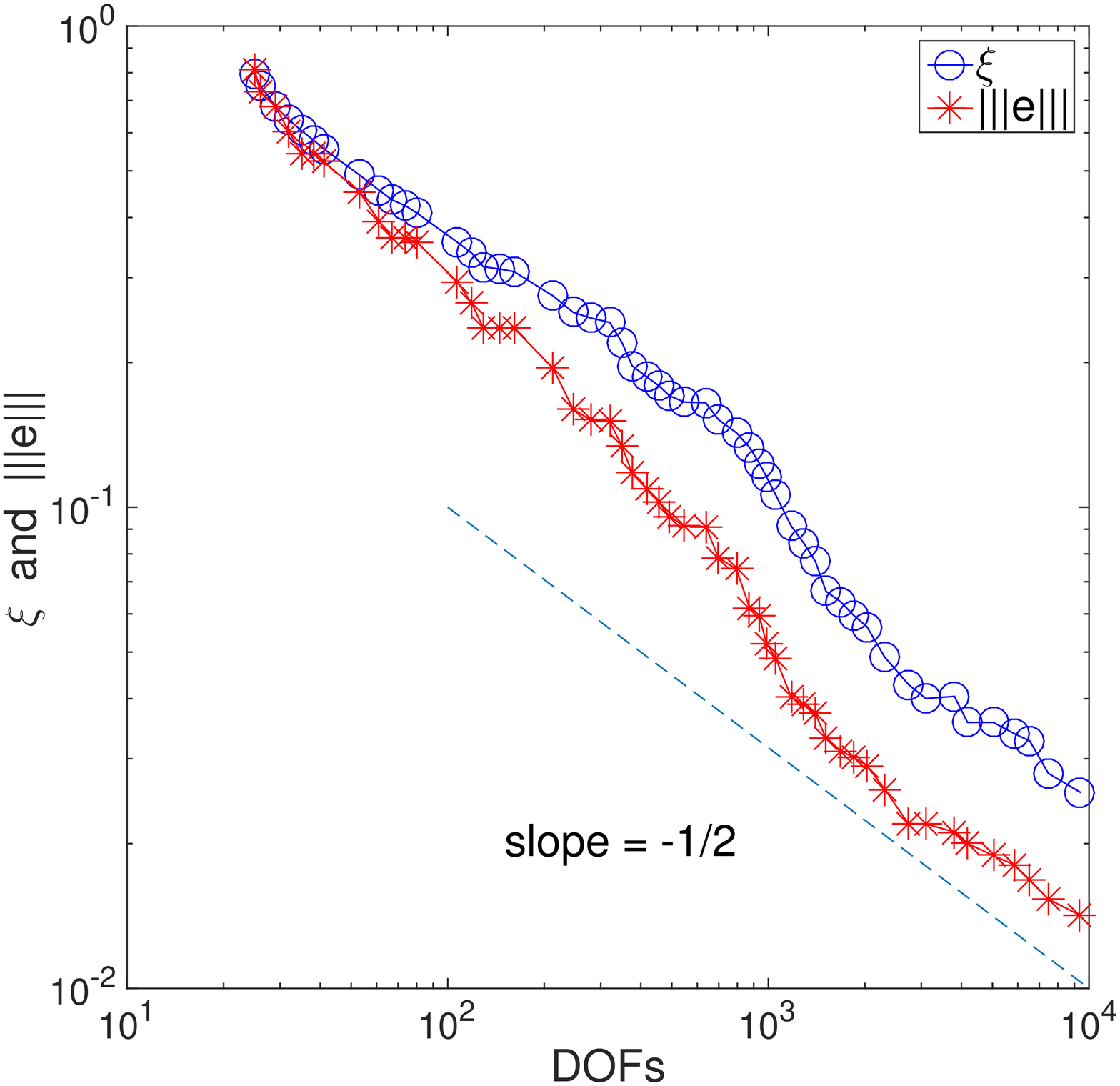}
\end{minipage}
\caption{Example 2 - error plot for $\eta$(left) and $\xi$(right)}
\label{fig:Ex2errorTOL01}
\end{figure}


\begin{table}
\caption{Example 2 - results for $\eta$ and $\xi$}
\label{tab:Ex2TOL01}
\begin{center}
\begin{tabular}{|c|c|c|c|}
\hline 
estimator &  DOFs & $\enorm{e}/\enorm{u}$ & eff-ind \\
\hline 
$\eta$ &  10987  & 9.8E-2 & 5.11 \\
\hline 
$\xi$ &  9383 & 9.9E-2 &  1.80 \\
\hline
\end{tabular}
\end{center}
\end{table}



\subsection{Test Problem 3} 
The exact solution is chosen as 
\[
    u(x,y) = \tanh\left(\epsilon^{-1/2}(x^2+y^2-\frac{1}{4})\right) 
    \quad \text{with} \quad \epsilon = 10^{-4},
\]
which displays an interior layer on a circle with radius $\frac{1}{2}$.

The stopping criterion for the adaptive mesh refinement is chosen as 
$\enorm{e} < 0.01\enorm{u}$.
The mesh generated by $\xi$ is shown in Figure \ref{fig:Ex23mesh} (right), 
where major refinements are along the interior layer.
In terms of accuracy or effectivity,
from Table \ref{tab:Ex3TOL001} and Figure \ref{fig:Ex3errorTOL001},
the same conclusion can be drawn as in Example 2:
the residual estimator is less accurate than the hybrid estimator
and its effectivity index varies more widely during the adaptive mesh refinement.
Figure \ref{fig:Ex3errorTOL001} shows that
the error starts to decay in optimal rate on very coarse meshes (DOFs$\approx$100), 
which demonstrates the robustness of the estimators $\eta$ and $\xi$.


\begin{figure}[htbp]
\centering
\begin{minipage}[t]{0.5\textwidth}
\centering
\includegraphics[width=.9\textwidth]{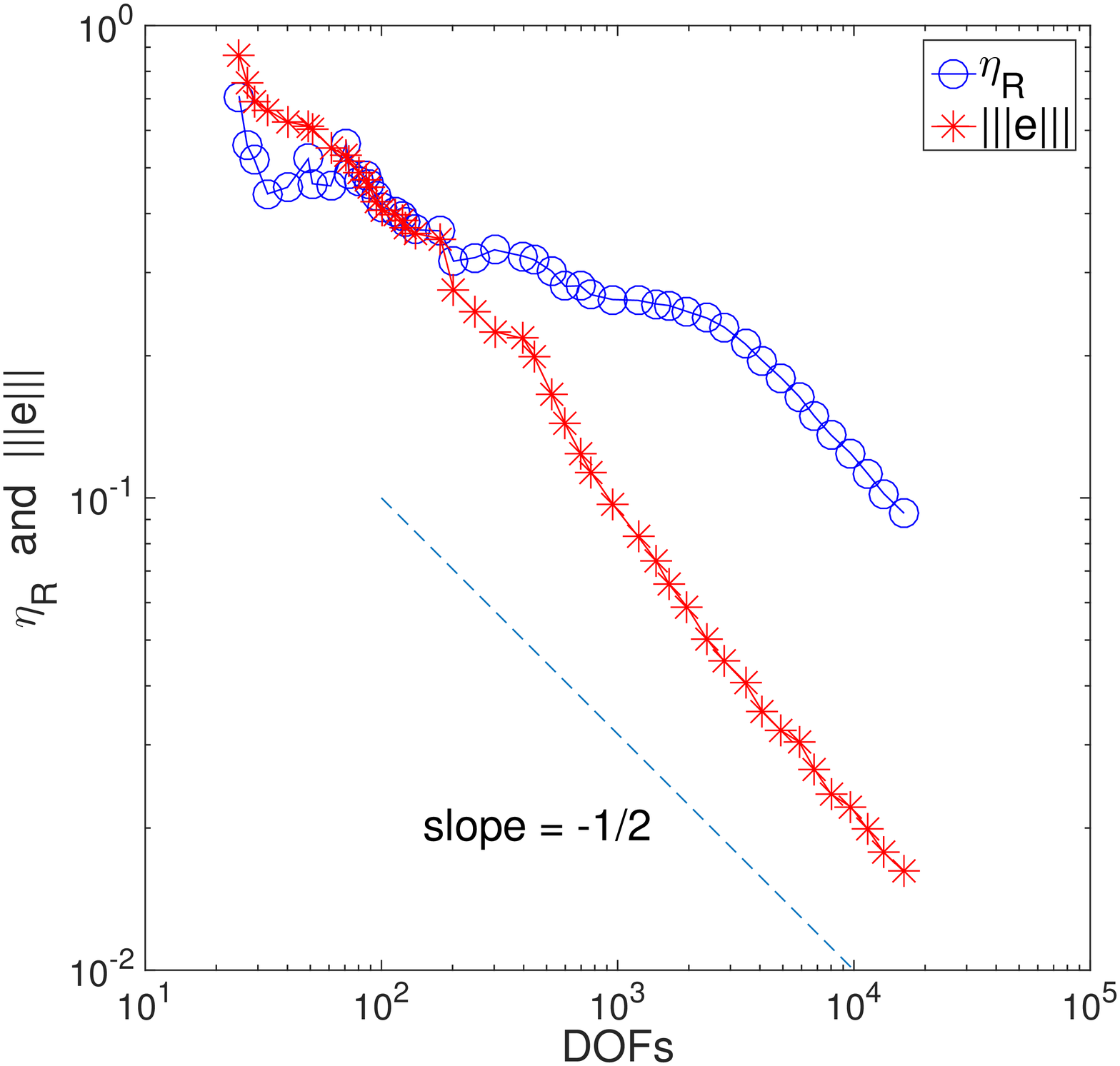}
\end{minipage}\hfill
\begin{minipage}[t]{0.5\textwidth}
\centering
\includegraphics[width=.9\textwidth]{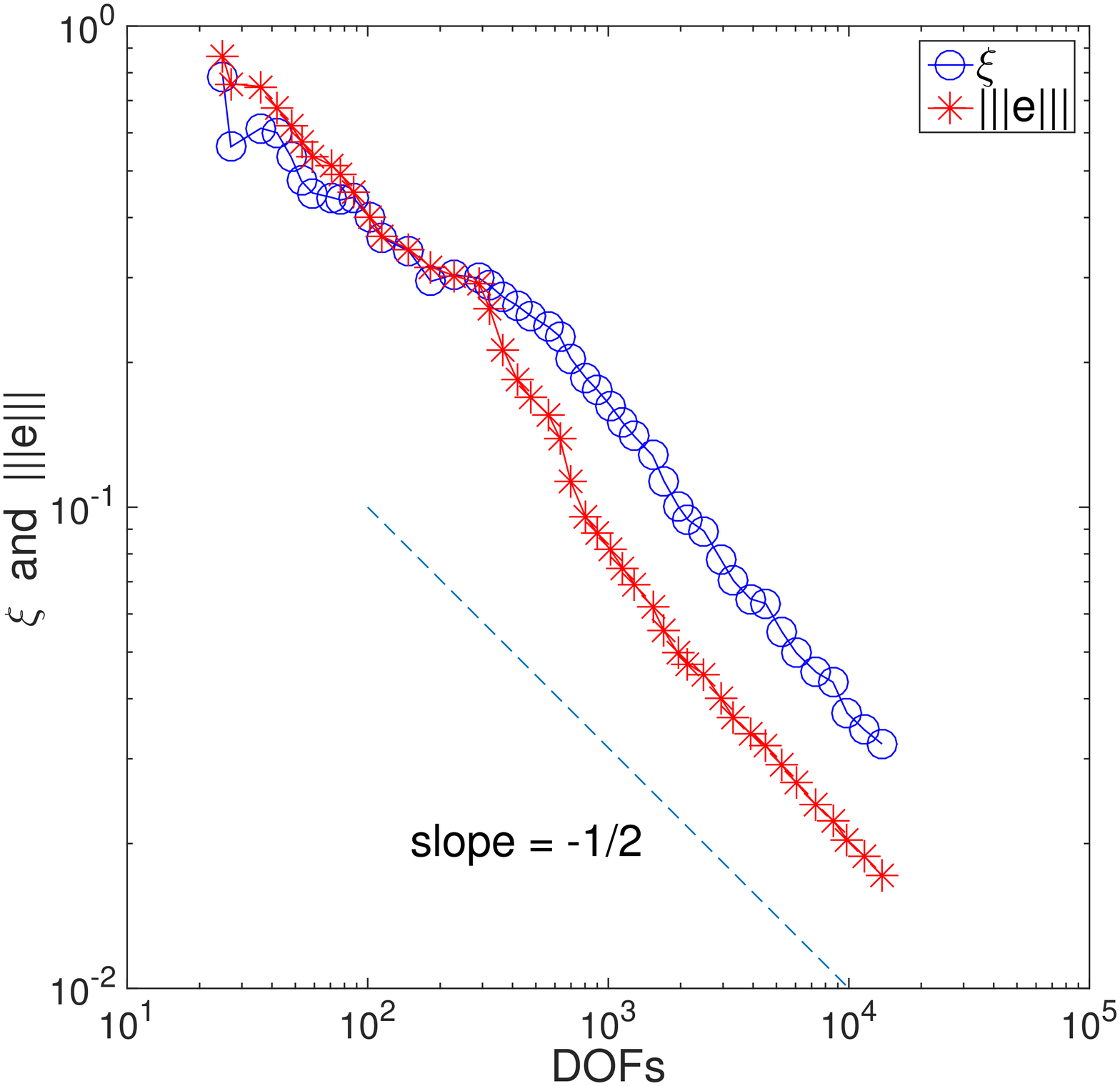}
\end{minipage}
\caption{Example 3 - error and estimator plot for $\eta$(left) and $\xi$(right)}
\label{fig:Ex3errorTOL001}
\end{figure}

\begin{table}
\caption{Example 3 - results for $\eta$ and $\xi$}
\label{tab:Ex3TOL001}
\begin{center}
\begin{tabular}{|c|c|c|c|}
\hline 
estimator &  DOFs & $\enorm{e}/\enorm{u}$ & eff-ind \\
\hline 
$\eta$ &  16217 & 9.2E-3 & 5.74 \\
\hline 
$\xi$ &  13664 & 9.8E-3 &  1.88 \\
\hline
\end{tabular}
\end{center}
\end{table}


\bibliography{cdfeng}
\bibliographystyle{plain}
\end{document}